\documentclass[11pt]{article}

\usepackage[utf8]{inputenc}

\usepackage{amssymb,MnSymbol,amsmath,amsthm,mathtools,enumitem,verbatim,bbm}

\setenumerate{label=\rm(\roman*),labelindent=.5em,widest=iv,leftmargin=*}
\setitemize{label=\textbullet,labelindent=.5em,leftmargin=*}

\newcommand\vv{z}
\newcommand\eps{\varepsilon}
\renewcommand\phi{\varphi}
\newcommand\dist{\mathrm{dist}}
\newcommand\R{\mathbb{R}}
\newcommand\N{\mathbb{N}}
\renewcommand\H{\mathcal{H}} 
\newcommand\Bc{\overline{B}} 
\newcommand\B{\overline{B}(0,1)} 
\newcommand\1{\mathbbm{1}}
\newcommand\Lip{\mathrm{Lip}}
\newcommand\Lips{\Lip} 

\newcommand\spt{\mathrm{spt}} 

\newcommand\NN{\mathcal{N}}
\newcommand{\spr}[2]{\langle#1,#2\rangle}

\newcommand\taux{\vartheta}

\newtheorem*{thm*}{Theorem}
\newtheorem{thm}{Theorem}[section]

\newtheorem{lem}[thm]{Lemma}
\newtheorem{cor}[thm]{Corollary}
\newtheorem{exm}[thm]{Example}

\theoremstyle{definition}
\newtheorem*{dfn*}{Definition}
\newtheorem{dfn}[thm]{Definition}
\newtheorem{rmk}[thm]{Remark}
\numberwithin{equation}{section}

\title{Cone unrectifiable sets and non-differentiability of
Lipschitz functions\thanks{The research leading to these results has received funding from the European Research Council / ERC Grant Agreement n.~291497. The first named author also acknowledges the support of the EPSRC grant EP/N027531/1 and of the 
National Science Foundation under Grant No. DMS-1440140 while the author was in residence at the Mathematical Sciences Research Institute in Berkeley, California, during the Fall 2017 semester.}}
\author{Olga Maleva and David Preiss}
\date{}

\begin{document}

\maketitle

\begin{abstract}
\noindent
We provide sufficient conditions for a set
$E\subset\R^n$ to be a non-universal differentiability set,
i.e.\ to be contained in the set of points of non-differentiability
of a real-valued Lipschitz function.
These conditions are motivated by
a description of the ideal generated by sets of
non-differentiability of Lipschitz self-maps of $\R^n$
given by Alberti, Cs\"ornyei and Preiss, which
eventually led to the result of Jones and Cs\"ornyei
that for every Lebesgue null set $E$ in $\R^n$
there is a Lipschitz map $f:\R^n\to\R^n$
not differentiable at any point of $E$, even though for $n>1$ and
for Lipschitz functions from $\R^n$ to $\R$ there exist
Lebesgue null universal differentiability sets.
\end{abstract}
\section{Introduction and main results}
A recent surge of interest in validity of Rademacher's theorem on
almost everywhere differentiability of
Lipschitz maps of $\R^n$ to $\R^m$ arose from several
new results and approaches. For infinite-dimensional Banach spaces
there were successful attempts to obtain its analogues
for the notion of G\^ateaux derivative,
for results and references see \cite[Chapter 6]{BL},
and
some results
for the stronger notion of Fr\'echet derivative
to which a recent monograph \cite{LPT} is devoted. In another direction,
Pansu~\cite{Pansu} obtained an almost everywhere result
for Lipschitz maps between Carnot groups, and
Cheeger~\cite{Cheeger} generalised Rademacher's theorem to Lipschitz
functions on metric measure spaces.

Here we contribute to this research in the direction
started by a result of \cite{p} that, in terms of the size of
differentiability sets
of real-valued Lipschitz functions on $\R^2$,
Rademacher's theorem is not sharp: there is a Lebesgue null
set in $\R^2$ containing points of differentiability of
any real-valued Lipschitz function on $\R^2$. Following \cite{MDo,MD},
where it was shown how unexpectedly small such sets may be, they
are now called universal differentiability sets. The
analogues of universal differentiability sets
were recently introduced and investigated
in the Heisenberg group~\cite{PS}.

The non-differentiability sets of Lipschitz maps $\R^n\to\R^m$,
$m\ge n$ were first completely described
in geometric measure theoretical terms in \cite{ACP}
(see \cite{ACP0,ACP1} for a published less formal
description), and then \cite{CJ} showed that this description gives
precisely the Lebesgue null sets in $\R^n$.
Hence Rademacher's theorem is sharp for maps into spaces of
the same or higher dimension. This result was complemented in \cite{pSp}
where it was proved that whenever $m<n$, there is a Lebesgue null
set in $\R^n$ containing points of differentiability of
any Lipschitz map $\R^n\to\R^m$.
We will return to the description originally introduced in \cite{ACP} later
as it forms the main starting point of what we do in the present paper.

The problem we address in this paper stems
from the above results: can one give a
geometric measure theoretical description of
non-differentiability sets of Lipschitz maps of
$\R^n$ to $\R$?
Notice that this is a question about sets and not
about measures: if we are given a $\sigma$-finite Borel
measure $\mu$ in $\R^n$ that is singular with respect to the
Lebesgue measure, we may use \cite{ACP} and \cite{CJ}
to find a Lipschitz $\mu$-almost everywhere non-differentiable
mapping $f=(f_1,\dots,f_n):\R^n\to\R^n$ and observe that for a random
choice of $\alpha_i\in(0,1)$ the real-valued function $\sum_{i=1}^n \alpha_i f_i$
is Lipschitz and $\mu$-almost everywhere non-differentiable.
This argument appears both in
\cite{ACP} and \cite{AM}, and moreover \cite{AM}
simplifies the general arguments
of \cite{ACP} in the special case of differentiability
$\mu$-almost everywhere; notice also that in this case even
the results of \cite{CJ} may be demonstrated by a more
accessible proof given in \cite{DR}
(which is based on different ideas).

We will now state our results and explain them in more detail. Their
proofs will be given in Section~\ref{proofs}. The short Section~\ref{examples}
contains two examples whose meaning will also be discussed here.

The main concept that we use is based on the notion of \emph{width}
that has been, together with several variants, introduced in \cite{ACP}.

\begin{dfn}\label{DF1}
Suppose $e\in\R^n\setminus\{0\}$ and $\alpha\in(0,1]$.
We let $C_{e,\alpha}$ be the cone
$\{x\in\R^n: \spr{x}{e}\ge\alpha\|x\|\|e\|\}$
and $\Gamma_{e,\alpha}$ the set of Lipschitz
curves such that $\gamma'(t)\in C_{e,\alpha}$ for almost every~$t\in\R$.
The \emph{$(e,\alpha)$-width} of an open set $G\subset \R^n$ is defined by
\begin{align}\label{D1}
w_{e,\alpha}(G) &
= \sup\{\H^1(G\cap\gamma(-\infty,\infty)): \gamma\in\Gamma_{e,\alpha}\},\\
\shortintertext{and of any $E\subset\R^n$ by}\label{D2}
w_{e,\alpha}(E) &= \inf\{w_{e,\alpha}(G): G\supset E, G\text{ is open}\}.
\end{align}
For the sake of completeness, when $e=0$ or $\alpha > 1$
we let $w_{e,\alpha}(E)=0$ for every $E\subset\R^n$.
Of course, these cases have no bearing on what we do.
\end{dfn}

Notice that, as \cite{ACP} points out, for constructions of Lipschitz
functions, where values of
$w_{e,\alpha}$ matter only for arbitrarily small $\alpha$,
the number $\alpha$ in Definition~\ref{DF1} may be replaced
by any quantity or function that may attain arbitrarily small positive values.
For example \cite{AM} replaces it by $\cos\alpha$, which is a bound
on the angle between $\gamma'(t)$ and $e$ and so is geometrically
natural, but
for us has the disadvantage that values of $\alpha$
that matter, namely those for which
the cone $C_{e,\alpha}$ is close to a half-space,
are close to $\pi/2$ rather than to zero.

Many variants of Definition~\ref{DF1} that are easily seen or shown
to be equivalent to the one given here may be found in
\cite[Definition~1.1 and Remark~1.2]{pi}.
A rather useful variant is that $\Gamma_{e,\alpha}$
may be defined as the collection of
$\gamma\in C^1(\R,\R^n)$ satisfying $\|\gamma'(t))\|=1$
and $\gamma'(t)\in C_{e,\alpha}$ for every~$t$.

Perhaps the most interesting
modification of Definition~\ref{DF1} comes from
a so far unpublished result of M\'ath\'e and allows
taking Borel sets~$G$ both in \eqref{D1} and \eqref{D2}.
It is not exactly equivalent with ours, but has the properties that
a set of $(e,\beta)$ width zero according to M\'ath\'e has $(e,\alpha)$
width zero according to the above definition for every $\alpha>\beta$,
and a set of $(e,\alpha)$
width zero according to the above definition
has $(e,\alpha)$ width zero according to M\'ath\'e.
We have not used this, since our constructions, like that of \cite{ACP},
use that width is determined by open sets, and so the only difference
would be that an appropriate
version of Definition~\ref{DF1} would be called M\'ath\'e's Theorem.

Terms like ``cone null'' have been used for sets
that are defined with the help of the notion of width. We follow this
trend in our main notion, introduced in Definition~\ref{def-cu}.
Before coming to it, we recall the main starting motivation
behind what we do,
namely the following definition from \cite{ACP1}
and a special case of
their result (proved in \cite{ACP}) which is most relevant for us.

\begin{dfn}[see \mbox{\cite[Definition 1.11]{ACP1}}]\label{ACP-D}
A map $\tau$ of a subset $E$ of $\R^n$ to the Grassmanian $G(k,n)$
is said to be a
\emph{$k$-dimensional tangent field} of $E$ if
\begin{equation}\label{ACP-DE}
w_{e,\alpha}\{x\in E: \tau(x)\cap C_{e,\alpha} =\{0\}\} =0
\text{ for every $e\in\R^n$ and $\alpha>0$.}
\end{equation}
\end{dfn}

Obviously, if $E$ is a $k$-dimensional embedded $C^1$ submanifold of $\R^n$,
its tangent field $\tau(x)$ satisfies \eqref{ACP-DE}. However,
the following theorem proved in \cite{ACP1,ACP}
shows that many non-smooth sets admit a $k$-dimensional tangent field.
Before stating it, we notice that Definition~\ref{ACP-D}
uses the value $\alpha$ in two different
meanings and so it is sensitive on the choice of the notion of width.
As a more detailed discussion of this will appear in \cite{ACP},
we just point out that the use of M\'ath\'e's width and
the width from Definition~\ref{DF1} are equivalent. The only case to treat
is when Definition~\ref{DF1} holds in M\'ath\'e's sense.
Assuming $w_{e,\alpha_k}\{x\in E: \tau(x)\cap C_{e,\alpha_k} =\{0\}\} =0$ in M\'ath\'e's sense for all $k\ge1$, where $0<\alpha_k<\alpha<1$ and $\alpha_k\to\alpha$, we conclude that in the sense of Definition~\ref{DF1} we have
$w_{e,\alpha}\{x\in E: \tau(x)\cap C_{e,\alpha_k} =\{0\}\} =0$ for all $k\ge1$. Hence writing
$\{x\in E: \tau(x)\cap C_{e,\alpha} =\{0\}\}
=\bigcup_{k=1}^\infty \{x\in E: \tau(x)\cap C_{e,\alpha_k} =\{0\}\}$,
we get
$w_{e,\alpha}\{x\in E: \tau(x)\cap C_{e,\alpha} =\{0\}\} =0$.

\begin{thm}[see \mbox{\cite[Theorem 1.12]{ACP1}}]\label{ACP-T}
A set $E\subset\R^n$ is contained in a non-differentiability set of a Lipschitz map
$\R^n\to\R^m$ for some, or any,
$m\ge n$ if and only if it admits an $(n-1)$-dimensional
tangent field. If $n=2$, this holds if and only of $E$ is Lebesgue null.
\end{thm}
As already mentioned, in the last assertion of Theorem~\ref{ACP-T} the assumption
$n=2$ was removed in \cite{CJ}. Notice also that the general case of
Theorem~\ref{ACP-T} says that the existence of $k$-dimensional tangent fields
is similarly related to the existence of functions that at every point of the set
can be differentiable in the direction of linear subspaces of dimension at most $k$ only.

Based on these results, we conjecture that sets of
non-differentiability of real-valued Lipschitz functions may be described
as those for which there is an $(n-1)$-dimensional
tangent field satisfying conditions that make it
in some sense closer to being ``genuinely'' $(n-1)$-dimensional.
We do not have a more precise conjecture, but a simple consequence of our
main result, Theorem~\ref{4}, is that sets for which there exists a continuous
$(n-1)$-dimensional tangent field are indeed sets of
non-differentiability of real-valued Lipschitz functions.

Since for the real-valued case only the tangent fields of codimension one
are relevant, we base our approach on an obvious variant of
Definition~\ref{ACP-D} that uses the normal fields instead of tangent fields.
More interestingly, having in mind conditions similar to continuity
of the normal field, we can define the normal vectors pointwise, while
in general no
pointwise definition of tangent fields from Definition~\ref{ACP-D} is known.
A highly interesting exception to this is
the special case when we are interested in measures
supported by a set admitting a $k$-dimensional tangent field where \cite{AM}
provides an interesting pointwise definition of the tangent field at almost every point.

Since our ``normal vectors'' are not exactly those orthogonal to
the tangent field from Definition~\ref{ACP-D}, we do not actually call them
``normal vectors'' and instead use just notation $\NN(E,x)$ for their collection.

\begin{dfn}\label{NV}
For $x\in E\subset\R^n$ let
\[\NN(E,x):=\{e \in\R^n: (\forall \eps>0)(\exists r>0) w_{e,\eps}(B(x,r)\cap E)=0\}.\]
\end{dfn}

\begin{rmk}\label{R1}
Although we will not use it directly, we remark that $\NN(E,x)$ is a linear subspace
of $\R^n$ for any $x\in E$. This follows from results on ``joining cones'' in~\cite{ACP}, but
we describe a quick argument based on
the result of M\'ath\'e. Since $\lambda \NN(E,x)=
\NN(E,x)$ for each $\lambda\in\R$, we conclude that every nonzero $e$ from the linear span of $\NN(E,x)$ can be written as
$e=\sum_{i=1}^k e_i$ where $e_i\in\NN(E,x)\setminus\{0\}$.
Suppose $\eps>0$ is fixed and $\gamma\in\Gamma_{e,\eps}$ belongs to $C^1(\R)$ and satisfies
$\|\gamma'(t)\|=1$ for all $t\in\R$ (cf.\ remarks after Definition~\ref{DF1}). Let $\alpha=\frac12\eps\|e\|/\sum_i\|e_i\|$ and find $\delta>0$
such that $w_{e_i,\alpha}(E\cap B(x,\delta))=0$ for each~$i$.
From Definition~\ref{DF1} we see that there is a Borel
(in fact $G_\delta$) set $G\supset E$ such that
$w_{e_i,\alpha}(G\cap B(x,\delta))=0$ for every~$i$.
Fix now any $t\in\R$ and notice that there exists an $i$ such that $\spr{\gamma'(t)}{e_i}\ge2\alpha\|e_i\|$. By continuity of $\gamma'$ there is a $\tau>0$ such that for this particular $i$ we have
$\spr{\gamma'(s)}{e_i}>\alpha\|e_i\|$ whenever $s\in(t-\tau,t+\tau)$.
Hence $w_{e_i,\alpha}(G\cap B(x,\delta))=0$ for this $i$ implies $|\gamma^{-1}(G\cap B(x,\delta))\cap(t-\tau,t+\tau)|=0$. Finally, existence of such $\tau>0$ for each $t\in\R$ allows us to conclude that $|\gamma^{-1}(G\cap B(x,\delta))|=0$. As this holds for every $\gamma\in\Gamma_{e,\eps}$, we get $w_{e,\eps}(G\cap B(x,\delta))=0$.
\end{rmk}

\begin{dfn}\label{def-cu}
A set $E\subset\R^n$ satisfying $\NN(E,x)\ne\{0\}$
for every $x\in E$
is said to be \emph{cone unrectifiable.}
\end{dfn}

\begin{rmk}\label{rmk-cu}
Of course any cone unrectifiable set is Lebesgue null.
A basic example of cone unrectifiable sets $E\subset\R^n$ is provided by those
for which $\NN(E,x)=\R^n$ for every $x\in E$.
Such sets are called \emph{uniformly purely unrectifiable.}
By the result of Andr\'as M\'ath\'e alluded to above these are precisely
those sets that are contained in a Borel $1$\nobreakdash-purely unrectifiable set,
i.e., in a Borel set $B$ whose intersection with
any $C^1$ curve has one-dimensional Hausdorff measure zero.
The arguments used to prove Remark~\ref{R1} simplify their definition
in another way: $E$ is uniformly purely unrectifiable if and only if
there is $0<\eta<1$ such that $w_{e,\eta}(E)=0$ for every unit vector~$e$ (this is used as a definition of a uniformly purely unrectifiable set in \cite{CPT}).
In Example~\ref{e4} we point out that a similar simplification of the notion
of cone unrectifiable sets is false:
given any $e\ne 0$ and $\eta\in(0,1)$,
we construct a set $E$ which does not satisfy the conclusions of the following Theorem~\ref{4c} but is of $C_{e,\eta}$-width zero.
\end{rmk}

We are now ready to state the main results of this paper.

\begin{thm}\label{4c}
If $E\subset\R^n$ is cone unrectifiable,
there is a Lipschitz function $f:\R^n\to\R$
that is non-differentiable at any point of $E$.
\end{thm}

There are various ways of stating
that non-differentiability of a function $f$
at a given point $x$ is rather strong.
The most usual one is by comparing
the upper and lower directional
derivatives of $f$ at $x$ defined by
\begin{align*}
{D^+}f(x;y)&:=\limsup_{t\searrow 0}\frac{f(x+ty)-f(x)}{t}\\
\shortintertext{and}
{D_+}f(x;y)&:=\liminf_{t\searrow 0}\frac{f(x+ty)-f(x)}{t},
\end{align*}
respectively. An even stronger non-differentiability
statement is obtained by showing that close to $x$,
$f$ may be approximated by many linear functions.
Our next result shows that the non-differentiability
statement of Theorem~\ref{4c} may be strengthened
in this way.

\begin{thm}\label{4}
For every cone unrectifiable set $E\subset\R^n$
and every $\eps>0$ there are a Lipschitz function $f:\R^n\to\R$
with $\Lip(f)\le 1+\eps$ and
a continuous function $u:E\to\{z\in\R^n: \|z\|\le\eps\}$ such that
\begin{equation}\label{E1}
\adjustlimits\liminf_{r\searrow 0}\sup_{\|y\|\le r}
\frac{|f(x+y)-f(x)-\spr{e+u(x)}y|}{r} =0
\end{equation}
whenever $x\in E$ and $e\in \NN(E,x)$ has $\|e\|\le 1$.
In particular,
\[{D^+}f(x;y)-{D_+}f(x;y)
\ge 2 \sup \bigl\{\spr ey: e\in \NN(E,x),\,\|e\|\le 1\bigr\}\]
whenever $x\in E$ and $y\in\R^n$.

Additionally, if $E$ is contained in $\{x:\omega(x)>0\}$,
where $\omega:\R^n\to[0,\infty)$ is continuous,
then $f$ may be chosen in such a way that $|f|\le\omega$.
\end{thm}

For a set $E$ admitting an $(n-1)$-dimensional continuous tangent
we obviously have
$\NN(E,x)\supset \tau(x)^\perp\ne\{0\}$. Hence
such sets are cone unrectifiable and so are sets of non-differentiability of
a real valued Lipschitz function.
More interestingly, having
countably many cone unrectifiable sets,
we may try to add the functions obtained from Theorem~\ref{4} to get
a function non-differentiable at the points of the union.
However, addition of non-differentiable functions
may create new points of differentiability.
To solve this problem we employ the idea that Zahorski \cite{Z} used in his
precise description of non-differentiability sets of
Lipschitz functions on the real line as $G_{\delta\sigma}$-sets
of measure zero; it is based on the simple observation
that the sum of a differentiable and a non-differentiable function
is non-differentiable. For this we need the function $f$ obtained
in Theorem~\ref{4} to be differentiable outside $E$, in other words, to have that $E$ coincides with the set of points where $f$ is not differentiable. This is
however not possible in general as shown in Example~\ref{e2}.
In the following two Corollaries we solve this
difficulty by making a special assumption
that the sets we consider~are~$F_\sigma$.

\begin{cor}\label{CAM}
Suppose $E=\bigcup_k E_k\subset\R^n$,
where $E_k$ are disjoint cone unrectifiable $F_\sigma$ sets,
and let $\NN_x:=\NN(E_k,x)\cap\B$ when $x\in E_k$.
Then there is a Lipschitz $f:\R^n\to\R$ such that
\begin{itemize}
\item
$f$ is differentiable at every $x\in\R^n\setminus E$;
\item
for every $x\in E$ there is $c>0$ such that for every $y\in\R^n$,
\[{D}^+f(x;y)-{D}_+f(x;y)\ge c \sup_{e\in \NN_x}\spr ey,\]
so, in particular, $f$ is not differentiable at $x$.
\end{itemize}
\end{cor}

If we are not interested in estimates of the difference
of the upper and lower derivatives, Corollary~\ref{CAM} gives
the following more naturally sounding statement.

\begin{cor}\label{CAMx}
For any $E\subset\R^n$
that is a countable union of cone unrectifiable $F_\sigma$ sets
there is a Lipschitz function $f:\R^n\to\R$
that is non-differentiable at any point of $E$ and is differentiable at any point of its complement $E^c$.
\end{cor}

The next Corollary~\ref{ACP-AM} contains the constructions of
$\mu$-almost everywhere non-differentiable functions from \cite{ACP}
and \cite[Theorem 4.1]{AM}. Given a Radon measure $\mu$ in $\R^n$, these authors
assign to $\mu$-a.a.\ points a linear subspace $T(x)$ of $\R^n$
that in certain sense represents the directions of curves on which $\mu$
is ``seen''. For \cite{ACP}, the definition of ``seen'' is exactly the assumption of
Corollary~\ref{ACP-AM} while \cite{AM} bases the definition on a
related but different property and shows in \cite[Lemma~7.5]{AM} that the assumption of Corollary~\ref{ACP-AM} is satisfied. It is, however, important to point out that both these references
define the linear space $T(x)$ which is in a natural sense smallest, and this allows them
to obtain also a counterpart to Corollary~\ref{ACP-AM} that every Lipschitz function
is $\mu$-a.e.\ differentiable in the direction of $T(x)$. We also notice
that the constructions of $\mu$-almost everywhere non-differentiable
Lipschitz functions have been strengthened in a different direction by \cite{MS}
where the authors produce functions that $\mu$-a.e.\ admit any blow-up behaviour permitted by
the differentiability results.

\begin{cor}\label{ACP-AM}
Let $\mu$ be a Radon measure on $\R^n$ and $T$ a
$\mu$-measurable map of
$\R^n$ to $\bigcup_{m=0}^n G(n,m)$
such that for every unit vector $e$ and $0<\alpha<1$, the set
$\{x: C_{e,\alpha}\cap T(x)=\nolinebreak\{0\}\}$
is the union of a $\mu$-null set and a set $E$ with $w_{e,\alpha}(E)=0$.
Then there is a real valued Lipschitz function $f$ on $\R^n$ such that
for $\mu$-a.e.\ $x\in\R^n$ there is $c>0$ such that
${D}^+f(x,v)-{D}_+f(x,v) \ge c \dist(v,T(x))$
for every $v\in\R^n$.
\end{cor}

Our final result deals with the \emph{uniformly purely unrectifiable} sets
introduced in Remark~\ref{rmk-cu}. For such sets the statement of
Theorem~\ref{4} concerning upper and lower derivatives is shown in \cite{ACP}.
We prove a stronger result, namely that these sets are
non-universal differentiability
sets in the strongest possible sense, which corresponds
to having $\eps=0$ in Theorem~\ref{4}.
However, in Example~\ref{e} we demonstrate that such an
improvement is specific to the case of uniformly purely unrectifiable
sets even when $E\subset\R^2$ is compact, for each $x\in E$ the set
$\NN(E,x)$ is a one-dimensional
linear subspace of $\R^2$
and the map $x\mapsto \NN(E,x)$ is continuous.

\begin{thm}\label{9}
For every uniformly purely unrectifiable set $E\subset\R^n$ there is
a real valued $1$-Lipschitz function $f$ on $\R^n$ such that
\begin{equation}\label{E2}
\adjustlimits\liminf_{r\searrow 0}\sup_{\|y\|\le r} \frac{|f(x+y)-f(x)-\spr{e}y|}{r} =0
\end{equation}
for every $x\in E$ and $e\in \R^n$ with $\|e\|\le 1$.
In particular,
${D}^+f(x;y) =\|y\|$ and ${D}_+f(x;y)=-\|y\|$
for every $x\in E$ and $y\in\R^n$.
\end{thm}

\section{Technical lemmas}

We will work in the space $\R^n$ equipped with the Euclidean norm $\|\cdot\|$.
Most of the notation we use is standard; the open and closed balls
will be denoted by $B(x,r)$ and $\Bc(x,r)$, respectively.
Since we will often need to use the distance of a point to the complement of
an open set, we will simplify the notation for it:
for an open $G\subset\R^n$ we let
\begin{equation}\label{dist-fun}
\rho_G(x):=\dist(x,\R^n\setminus G).
\end{equation}

As usual,
the Lipschitz constant of a real-valued function $f$ defined on a set $E\subset\R^n$
is the smallest constant $\Lip(f,E)\in[0,\infty]$, or just $\Lip(f)$ when $E=\R^n$,
such that
$|f(y)-f(x)|\le \Lip(f,E)\|y-x\|$ for all $x,y\in E$, and
functions with $\Lip(f)\le c$ will be termed $c$-Lipschitz.
The space of Lipschitz functions on~$\R^n$, those for which $\Lip(f)<\infty$, will be denoted by $\Lips(\R^n)$.
We will also use the pointwise Lipschitz constants defined by
$\Lip_x(f):=\limsup_{y\to x}|f(y)-f(x)|/\|y-x\|$ and often use the
following well known fact.

\begin{lem}\label{LL}
For any $f:\R^n\to\R$, it holds that $\Lip(f)=\sup_{x\in\R^n} \Lip_x(f)$.
\end{lem}
\begin{proof}
It suffices to handle the case $n=1$ when it follows,
for example, from the
considerably more general Theorem 4.5 in
\cite[Chapter IX]{saks}.
\end{proof}

The following lemma allows us to modify the functions we are constructing
so that they become smooth on suitable subsets of $\R^n$.
A similar lemma is proved in \cite{ACP},
and in \cite{AM}, although the authors of the latter paper could have used
more general
\cite[Theorem 1]{A} or \cite[Corollary 16]{HJ}.
We need a slightly more technical variant of results from these references.

Recall that for any collection $\cal B$ of open sets in $\R^n$
there is a
$C^\infty$ partition of unity of order $n$ subordinated to it,
that is a collection of $C^\infty$ functions
$\phi_k:\R^n\to[0,1]$, $k=1,2,\dots$, such that
\begin{itemize}
\item
each $\spt(\phi_k)$ is a compact subset of some $B\in\cal B$,
\item
$\sum_k\phi_k(x)=1$ for every $x\in\bigcup\cal B$,
\item
for each $x\in\bigcup\cal B$ there is $r>0$ such that
$B(x,r)\cap\spt(\phi_k)\ne \emptyset$ for at most
$n+1$ values of $k$.
\end{itemize}

\begin{lem}\label{AL}
Suppose $H\subset\R^n$ is open, $g:\R^n\to\R$ is Lipschitz,
$\Phi: H\to\R^n$
and $\xi: H\to[0,\infty)$ are continuous and bounded, and
$\|g'(x)-\Phi(x)\|\le\xi(x)$ for almost every $x\in H$.
Then for every continuous $\omega: H\to(0,\infty)$
there is a Lipschitz function $f:\R^n\to\R$ such that
\begin{enumerate}
\item\label{2z.1}
$f(x)=g(x)$ for $x\notin H\cap\{\xi>0\}$ and
$|f(x)- g(x)| \le \omega(x)$ for $x\in H$;
\item\label{2z.2}
$f\in C^1(H)$ and
$\|f'(x)-\Phi(x)\|\le \xi(x)(1+\omega(x))$ for $x\in H$;
\item\label{2z.3}
$\Lip(f)\le \max(\Lip(g),\sup_{x\in H}(\|\Phi(x)\|+\xi(x)(1+\omega(x))))$.
\end{enumerate}
\end{lem}

\begin{proof}
Let $U:=H\cap\{\xi>0\}$, extend $\xi$ and $\omega$ to
(possibly discontinuous) functions defined on all of $\R^n$
by letting $\xi(x)=\omega(x)=0$ for $x\notin H$ and let
$\omega_0(x):=\tfrac12\min(1,\xi(x)\omega(x),\omega(x),\rho_U^2(x))$.
Let $\mathcal{B}$
be the family of balls $B(x,r)$
such that $x\in U$ and $r<\omega_0(x)$. Choose
$(\phi_k)_{k\ge1}$ forming a locally finite $C^\infty$ partition of unity
on $U$ subordinate to~$\mathcal{B}$, and denote $m_k=1+\|\phi_k'\|_\infty$.

As, for example, in \cite[Appendix C.4]{evans},
let $\eta$ be the standard $C^\infty$-smooth mollifier in~$\R^n$
and define $\eta_s(x):=\eta(x/s)/s^n$. For each $k$
choose $s_k>0$ small enough
so that the convolution $f_k=g *\eta_{s_k}$
satisfies for every $x\in \spt(\phi_k)$,
\begin{itemize}
\item
$|f_k(x)-g(x)|\le 2^{-k-1} m_k^{-1}\omega_0(x)$;
\item
$\|f_k'(x)-\Phi(x)\|\le \xi(x)+\omega_0(x)$.
\end{itemize}

Define $f:\R^n\to\R$ by $f(x)=\sum_{k} f_k(x)\phi_k(x)$ for $x\in U$
and $f(x)=g(x)$ for $x\notin U$.
Since each $f_k\phi_k$ is in $C^1(U)$, we have $f\in C^1(U)$.
Also, for all $x\in\R^n$,
\begin{equation}\label{U.1}
|f(x)-g(x)| \le \omega_0(x)
\end{equation}
since for $x\notin U$ both sides are zero, and for $x\in U$,
\begin{align*}
|f(x)-g(x)|&\le \sum_k |f_k(x)-g(x)|\phi_k(x)
\le\sum_k \omega_0(x)\phi_k(x) \le \omega_0(x).
\end{align*}
Since $\omega_0\le\omega$ and $\omega_0(x)=0$ for $x\notin U$,
\ref{2z.1} holds.

We show that $f$ is differentiable at every $x\in H$ and
\begin{equation}\label{U.2}
\|f'(x)-\Phi(x)\|\le \xi(x)+2\omega_0(x).
\end{equation}
To see this
for $x\in U$, we use $\sum_k \phi_k(x)=1$ and $\sum_k \phi_k'(x)=0$ to infer that
\begin{align*}
f'(x)-\Phi(x)&= \sum_k (f_k'(x)-\Phi(x))\phi_k(x)
+ \sum_k (f_k(x)-g(x))\phi_k'(x),\\
\shortintertext{hence}
\|f'(x)-\Phi(x)\|
&\le \sum_k \|f_k'(x)-\Phi(x)\|\phi_k(x)
+ \sum_k |f_k(x)-g(x)|\|\phi_k'(x)\|\\
&\le \sum_k (\xi(x)+\omega_0(x))\phi_k(x)+\sum_k 2^{-k-1}\omega_0(x)\\
&\le \xi(x)+2\omega_0(x).
\end{align*}
To see \eqref{U.2} for $x\in H\setminus U$, we infer from
the assumptions on $g,\Phi$ and $\xi$ that
$g$ is differentiable at $x$ and $g'(x)=\Phi(x)$.
Since $(f-g)'(x)=0$ because
\eqref{U.1} gives
$|f(y)-g(y)|\le\omega_0(y)\le\rho_U^2(y)\le\|y-x\|^2$
for all $y\in\R^n$,
we get that $f$ is differentiable at $x$ and
$f'(x)=g'(x)=\Phi(x)$.

Clearly, \eqref{U.2} and the inequality
$2\omega_0(x)\le\xi(x)\omega(x)$ show
the second statement of \ref{2z.2}.

To prove \ref{2z.3}, we infer from \eqref{U.1}
that $\Lip_x(f)\le\Lip(g)$ for $x\in\R^n\setminus U$,
and from \eqref{U.2} that
\[\Lip_x(f)\le \sup_{y\in U}(\|\Phi(y)\|+\xi(y)+2\omega_0(y))
\le \sup_{y\in H}(\|\Phi(y)\|+\xi(y)(1+\omega(y)))\]
for $x\in U$. Thus \ref{2z.3} holds by Lemma~\ref{LL}
and, since its right side
is finite, we also see that $f$ is Lipschitz.

We already know that $f$ is differentiable
at every $y\in H$ and $f'$ is continuous at every $y\in U$.
If $y\in H\setminus U$, \eqref{U.2}
shows that
$\lim_{x\to y} (f'(x)-\Phi(x)) = 0$. Since
$\Phi$ is continuous at $y$, it follows that
$f'$ is continuous at $y$.
Hence $f\in C^1(H)$, which is the last statement
we needed to prove.
\end{proof}

The next simple Lemma is used to show that the functions we construct
may be approximated by linear ones in the way required in
equation~\eqref{E1}
of our main result, Theorem~\ref{4}.

\begin{lem}\label{3x}
Suppose that $H\subset\R^n$ is open, $g:\R^n\to\R$ belongs to $C^1(H)$,
$\omega:\R^n\to[0,\infty)$ is continuous and strictly positive on $H$,
and $\eta\in(0,1]$.
Then there is a function $\xi:\R^n\to[0,\infty)$ such that
\begin{enumerate}
\item\label{3x.0}
$\xi\in C(\R^n,[0,\infty))\cap C(H,(0,\infty))$
and $\xi\le\tfrac12\omega$;
\item\label{3x.2}
if $x\in H$ and $h:\R^n\to\R$ satisfies
$|h-g|\le2\xi$, there is $0<r<\omega(x)$
such that $|h(x+y)-h(x)-\spr{g'(x)}{y}|\le\eta r$
whenever $\|y\|\le r$.
\end{enumerate}
\end{lem}

\begin{proof}
Let $\Psi$ be the set of functions
$\psi:\R^n\to[0,\infty)$ satisfying $\Lip(\psi)\le 1$,
$0\le \psi\le\tfrac12\min(\rho_H,\omega,1)$,
and $\|g'(y)-g'(z)\|\le\tfrac12\eta$
whenever $x\in H$ and $\max(\|y-x\|,\|z-x\|)<\psi(x)$.
Since $0\in\Psi$, $\phi(x):=\sup\{\psi(x): \psi\in\Psi\}$
is well-defined. We also have $\phi\in\Psi$ since for any
$x,y,z$ satisfying $x\in H$ and $\max(\|y-x\|,\|z-x\|)<\phi(x)$
there is $\psi\in\Psi$ such that
$\max(\|y-x\|,\|z-x\|)<\psi(x)$ and hence
$\|g'(y)-g'(z)\|\le\tfrac12\eta$.

Let $x\in H$. Since both $\rho_H$ and $\omega$ are continuous and
strictly positive at $x$, there is $\eps>0$ such that
$\tfrac12\min(\rho_H,\omega,1)>\eps$ on $B(x,\eps)$.
Then the function $\psi_{\eps,x}(y):=\max(0,\eps-\|y-x\|)$
satisfies $\psi_{\eps,x}=0$ outside $B(x,\eps)$ and
$0\le\psi_{\eps,x}(y)\le \eps \le \tfrac12\min(\rho_H(y),\omega(y),1)$
for $y\in B(x,\eps)$.
Hence $\psi_{\eps,x}$ belongs to $\Psi$ and
we infer that
$\phi(x)\ge\psi_{\eps,x}(x)=\eps>0$.
Consequently, $\phi$ is strictly positive on $H$.
Furthermore,
\[|g(x+y)-g(x)-\spr{g'(x)}{y}|\le\|y\|\sup_{z\in B(x,\|y\|)} \|g'(z)-g'(x)\|
\le \tfrac12\eta\|y\|\]
whenever $x\in H$ and $\|y\|< \phi(x)$.

Letting $\xi(x):=\tfrac1{12}\eta\phi(x)$, we see that
\ref{3x.0} holds. To prove \ref{3x.2}, given $x\in H$, we let
$r:=\phi(x)$, observe that $0<r<\omega(x)$ and use that
$\Lip(\xi)\le\tfrac1{12}\eta$ and $\xi(x)=\frac1{12}\eta r$ to estimate
\begin{align*}
|h(x+y)-h(x)&-\spr{g'(x)}{y}|\\
&\le
2\xi(x+y)+2\xi(x) +|g(x+y)-g(x)-\spr{g'(x)}{y}|\\
&\le
4\xi(x) + 2\Lip(\xi)\|y\|+ \tfrac12\eta\|y\|\\
&\le \tfrac13 \eta r+\tfrac16\eta\|y\|+\tfrac12\eta\|y\|
\le \eta r
\end{align*}
whenever $\|y\| < r =\phi(x)$, and so whenever $\|y\|\le\xi(x)$.
\end{proof}

The following Lemmas~\ref{1} and \ref{2} modify corresponding lemmas
from \cite{ACP} in a way suitable for our applications.
A special version of Lemma~\ref{1}, which
does not suffice for our purposes, can be found also in
\cite[Lemmas~4.12--4.14]{AM}. Since \cite{ACP} is not
yet available, we provide full proofs.

\begin{lem}\label{1}
Given $\eps>0$ there is $\taux\in(0,1)$ such that the following holds.
For every $E\subset\R^n$, every unit vector $e\in\R^n$
such that $w_{e,\taux}(E)=0$
and every
continuous $\omega:\R^n\to[0,\infty)$ which is strictly positive on $E$,
there is a Lipschitz function $g:\R^n\to\R$ such that
$0\le g\le \omega$,
$\Lip(g)\le 1+\eps$
and there is an open set $H\supset E$ contained in $\{\omega>0\}$ such that
$\|g'(x)-e\|\le \eps$ for
Lebesgue almost all $x\in H$.
\end{lem}

\begin{proof}
Let $\taux=\sin\beta$, where $0<\beta<\pi/2$, be such that $\tan\beta<\eps/2$.
Denote $G:=\nobreak\{x: \omega(x)>0\}$ and choose $\phi_k\in C^\infty(\R^n)$, $k\ge1$, with
compact support contained
in $G$ that form
a locally finite partition of unity on $G$.
Let $\eps_k>0$ be such that
$\sum_k \eps_k\|\phi_k'\|<\eps/2$ and
$\eps_k\phi_k(x)\le 2^{-k} \min(1,\rho_G^2(x),\omega(x))$ for each $k\ge1$ and all $x\in\R^n$.

Using values $\eps_k$ which we have just defined, find open sets $G_k$ such that $G\supset G_k\supset E$ and $w_{e,\taux}(G_k)<\eps_k$.
For each $x\in\R^n$ we put
\begin{equation}\label{g}
g_k(x) := \sup \Bigl\{\H^1\bigl(G_k\cap \gamma(-\infty,b]\bigr)-s :
\gamma\in\Gamma_{e,\taux},\,s \ge 0,\,\gamma(b)=x+se\Bigr\}
\end{equation}
and show that
\begin{enumerate}
\item\label{1a.1}
$0 \le g_k(x) \le \eps_k$;
\item\label{1a.1a}
$|g_k(x+y)-g_k(x)|\le\|y\|\tan\beta$ when $y$ is perpendicular to $e$;
\item\label{1a.1b}
$g_k(x)\le g_k(x+re)\le g_k(x)+r$ for every $r>0$;
\item\label{1a.1c}
$g_k(x + re) = g_k(x) + r$ when $[x, x + re]\subset G_k$;
\item\label{1a.2}
$g_k$ is a Lipschitz function and $\Lip(g_k)\le 1+\tan\beta$;
\item\label{1a.3}
$\|g_k'(x) -e\|\le\tan\beta$ for almost every $x\in G_k$.
\end{enumerate}

The first inequality in \ref{1a.1}
is obvious by considering in \eqref{g},
$s=0$ and any $\gamma\in\Gamma_{e,\taux}$ with $\gamma(b)=x$,
and the second is immediate from $w_{e,\taux}(G_k)<\eps_k$.

If $y\ne 0$ is orthogonal to $e$, and $\gamma$, $b$, $s$ come from \eqref{g},
we let $r:=\|y\|$ and $\hat y:=y/r$ and redefine $\gamma$
on $(b,\infty)$ by $\gamma(b+t)=\gamma(b)+ (t\cot\beta)\hat y+te$ for $t>0$.
Using \eqref{g} for~$g_k(x+y)$ with $b':=b+r\tan\beta$ and $s':= s+ r\tan\beta$,
we get
\[g_k(x+y)\ge g_k(x) -r\tan\beta=g_k(x) -\|y\|\tan\beta.\]
To get a lower estimate for $g_k(x)$ apply the above to the vector $-y$ added to $x+y$: \[g_k(x)=g_k(x+y-y)\ge g_k(x+y)-\|y\|\tan\beta.\] This verifies \ref{1a.1a}.

Now consider $x'=x+re$ where $r>0$. Since any $\gamma$ used for $x'$ may be used for $x$ with $\gamma(b)=x+(r+s)e$,
we get $g_k(x)\ge g_k(x')-r$. For the rest of \ref{1a.1b} and for \ref{1a.1c},
note that as any $\gamma$ used in \eqref{g}
for $x$ may be redefined by letting $\gamma(b+t)=x+se+te$ for $t\ge 0$,
we get
\[g_k(x')\ge \H^1(G_k\cap \gamma(-\infty,b+r])-s\ge \H^1(G_k\cap \gamma(-\infty,b])-s\]
for all $\gamma$ satisfying \eqref{g}, so $g_k(x')\ge g_k(x)$, and this verifies \ref{1a.1b}.
If $[x,x']=[x,x+re]\subset G_k$ and $r\le s$, the same argument shows that
\[g_k(x')\ge \H^1(G_k\cap \gamma(-\infty,b+s])-(s-r)\ge\bigl(\H^1(G_k\cap \gamma(-\infty,b])-s\bigr)+r,\]
and if $r>s$, then
\begin{align*}
g_k(x')\ge \H^1(G_k\cap \gamma(-\infty,b+r])
&= \H^1(G_k\cap \gamma(-\infty,b+s]) +(r-s)\\
&\ge\bigl(\H^1(G_k\cap \gamma(-\infty,b])-s\bigr)+r
\end{align*}
for all such $\gamma$. Hence in both cases $g_k(x')\ge g_k(x)+r$, which, together with \ref{1a.1b}, implies equality in \ref{1a.1c}.

The statements \ref{1a.1a}--\ref{1a.1c}
imply that $g_k$ is Lipschitz and for almost every~$x$,
$0\le Dg_k(x;e)\le 1$, the equality $Dg_k(x;e)=1$ is satisfied for $x\in G_k$ and
$|Dg_k(x;y)|\le \|y\|\tan\beta$ for $y$ perpendicular to $e$.
This gives both \ref{1a.2} and~\ref{1a.3}.

Let $\displaystyle g:=\sum_{k=1}^\infty g_k\phi_k$. Since
by \ref{1a.1} one has $0\le g_k\phi_k\le 2^{-k}\min(1,\rho_G^2,\omega)$ for every $k\ge1$,
we conclude that $0\le g\le \omega$ and $\Lip_x(g)=0$ for $x\notin G$.
Since the sum defining $g$ is locally finite,
$g$ is locally Lipschitz on $G$
and by~\ref{1a.2} and \ref{1a.1} for almost every $x\in G$,
\[\|g'(x)\|
\le \sum_k \|g_k'(x)\|\phi_k(x) + \sum_k g_k(x)\|\phi_k'(x)\|
\le 1+\tan\beta + \sum_k \eps_k \|\phi_k'\| \le 1+\eps.\]
Hence $\Lip_x(g)\le 1+\eps$ for every $x\in G$,
and we infer from Lemma~\ref{LL}
that $\Lip(g)\le 1+\eps$.

Let $H :=\bigcap_k U_k$, where $U_k:=(G\setminus\spt(\phi_k))\cup G_k$ are open.
Then $E\subset\bigcap_k G_k\subset H\subset G$ and $H$ is open because
the complements of the $U_k$ in $G$ are closed in $G$
and their collection is
locally finite in $G$ since $G\setminus U_k\subset\spt(\phi_k)$. Finally, by \ref{1a.3}
for almost every $x\in H$,
\[\|g'(x)-e\|
\le \sum_k \|g_k'(x)-e\|\phi_k(x) + \sum_k g_k(x)\|\phi_k'(x)\|
\le \tan\beta + \sum_k \eps_k \|\phi_k'\| <\eps.\qedhere\]
\end{proof}

\begin{dfn}\label{DF.T}
Since we will need to use Lemma~\ref{1}
for several values of $\eps$ at the same time,
we introduce a function $\taux:(0,\infty)\to(0,\infty)$
such that $\taux(\sigma)$ is the value of $\taux$
from Lemma~\ref{1} for $\eps=\tfrac17\sigma$.
\end{dfn}

\begin{lem}\label{2}
Suppose $E\subset\R^n$,
the functions $\omega:\R^n\to[0,\infty)$ and $\phi:\R^n\to[0,1]$ are continuous,
$\omega>0$ on $E$,
$e\in\R^n$, $\sigma>0$ and $w_{e,\taux(\sigma)}(E\cap \{\phi>0\})=0$.
Then there exist functions $f:\R^n\to\R$ and $\psi:\R^n\to[0,1]$ and an open set $H\subset\R^n$ such that
\begin{enumerate}
\item\label{2.1}
$E\subset H\subset\{x:\omega(x)>0\}$ and
$f\in \Lips(\R^n)\cap C^1(H)$;
\item\label{2.2}
$|f(x)|\le\omega(x)\|e\|$ for all $x\in\R^n$ and $f(x)=0$ when $\phi(x)=0$;
\item\label{2.3}
$\|f'(x)-\psi(x)e\|\le\sigma\1_{\{\omega>0\}}(x)\1_{\{\phi>0\}}(x)\|e\|$
for almost all~$x\in\R^n$;
\item\label{2.4}
$0\le\psi(x)\le\phi(x)\1_{\{\omega>0\}}(x)$ for $x\in\R^n$ and
$\psi(x)=\phi(x)$ for $x\in H$.
\end{enumerate}
\end{lem}

\begin{proof}
If $e=0$ or $\sigma\ge 1$,
it suffices to let $f:=0$, $\psi:=\phi$ and $H:=\{\omega>0\}$.
So we assume $\|e\|=1$ and $\sigma<1$, let $\eps:=\sigma/7$ and
pick an integer $k\in[6/\sigma,7/\sigma]$.

Let $\omega_0:=\tfrac12\min(1,\omega)$,
$G_0:=\{\omega>0\}$, $H_0:=G_0\cap\{\phi>0\}$
and, whenever $H_{i-1}$ has been defined for some $i=1,\dots,k$,
let $G_i:=H_{i-1}\cap \{\phi>i/k\}$ and use Lemma~\ref{1} with continuous $\omega_i(x)=\frac12\min(\omega,\rho_{G_i}^2)$, where $\rho_{G_i}$ is defined by \eqref{dist-fun},
to find a Lipschitz function $g_i:\R^n\to\R$ and
nested open sets $H_i\subset G_i\subset H_{i-1}$ such that for each $1\le i\le k$,
\begin{enumerate}[label=(\alph*)]
\item\label{2p.2}
$\Lip(g_i)\le 1+\eps$ and $|g_i|\le \frac12\min(\omega,\rho^2_{G_i})$;
\item\label{2p.3}
$G_i\supset H_{i}\supset G_i \cap E$
and $\|g_i'(x)-e\|\le\eps$ for a.e.~$x\in H_{i}$.
\end{enumerate}

Let $g:=\frac1k\sum_{i=1}^{k} g_i$. Then by \ref{2p.2},
$\Lip(g)\le 1+\eps$
and $|g|\le \tfrac12\min(\omega,\rho_{G_1}^2)$.
For any $x\in G_0$ find the biggest $j=j(x)\in\{0,1,\dots,k\}$
with $x\in G_j$; since $G_k=\emptyset$, we have $j(x)\le k-1$.
Define $\psi(x)=\min((j(x)+2)/k,\phi(x))$;
and for $x\notin G_0$ let
$\psi(x)=0$. Clearly, $0\le\psi\le\phi\1_{G_0}$ on $\R^n$,
which is the first statement of \ref{2.4}.
For any $x\in H_0$ it holds $\psi(x)\in\Bigl(\frac{j(x)}k,\frac{j(x)+2}{k}\bigr]$,
i.e.\ $0<\psi(x)-j(x)/k\le 2/k$. Define now
\[H=\bigcup_{j=0}^k \{x \in H_j: \phi(x)<(j+2)/k\}\]
and notice that $\psi(x)=\phi(x)$ whenever
$x\in H$. Indeed, if $x \in H_j$ is such that $\phi(x)<(j+2)/k$, then $H_j\subset G_j$ implies $j(x)\ge j$, so $\frac{j(x)+2}{k}\ge\frac{j+2}{k}>\phi(x)$, hence by definition $\psi(x)=\phi(x)$, and this verifies~\ref{2.4}.
Also, $E\subset H$ since $E\subset \bigcup_{j=0}^{k-1}(H_j\setminus G_{j+1})$ from \ref{2p.3}, and for $x\in H_j\setminus G_{j+1}$ we have $j(x)=j$ and so
$\phi(x)\le(j+1)/k<(j+2)/k$. Since it is clear that $H$ is open and $\omega>0$ on $H$ because $H\subset G_0$, we conclude that the first part of \ref{2.1} is satisfied for $E$, $H$ and $\omega$. We are now left to define the Lipschitz function $f$ and verify the remaining part of \ref{2.1}, and also \ref{2.2} and \ref{2.3}.

Note that for almost all $x\in G_1$ (where $\phi>1/k$), all $g_i$ are differentiable at $x$ and the estimate in \ref{2p.3} is satisfied whenever $x\in H_i$ and $1\le i\le k$. Consider any such $x\in G_1$.
To estimate $g'(x)$, notice that for such $x$ we have $j=j(x)\ge1$ and
\begin{itemize}
\item
if $1\le i<j$, then $x\in H_i$ and so $\|g_i'(x)- e\|\le \eps$ by \ref{2p.3};
\item
if $i\ge j+1$, then $x\notin G_i$ and so $g_i'(x)=0$ by \ref{2p.2}.
\end{itemize}
Hence, for almost all $x\in G_1$
\begin{align*}
\|g'(x)-\psi(x)e\|
&\le\|g'(x)-\tfrac{j-1}{k}e\|+\tfrac3k\|e\|\\
&\le\tfrac1k(\sum_{i=1}^{j-1}\|g_i'(x)-e\| +\|g_j'(x)\|)+\tfrac3k\\
&\le\eps +\tfrac4k\le \tfrac5k\le5\phi(x).
\end{align*}
Since $g'(x)=0$ outside $G_1$, we get
$\|g'(x)-\psi(x)e\|=\psi(x)$ for $x\notin G_1$.
Using that $\psi=0$ outside $H_0$ and $\psi\le\phi\le\tfrac1k$
for $x\in H_0\setminus G_1$
we infer that $\|g'(x)-\psi(x)e\|\le \min(\tfrac1k,\phi(x))$ outside $G_1$
and conclude
$\|g'(x)-\psi(x)e\|\le 5\min(\tfrac1k,\phi(x))$
for almost all $x\in\R^n$.

Using Lemma~\ref{AL}
with $\Phi(x):=\psi(x)e$, $\xi(x):=5\min(\tfrac1k,\phi(x))$
and $\hat\omega(x)=\tfrac15\min(\omega(x),\phi(x),\rho_H^2)$,
we find Lipschitz $f:\R^n\to\R$ such that $f\in C^1(H)$,
$|f(x)- g(x)| \le \hat\omega(x)$ and
$\|f'(x)-\psi(x)e\|\le \xi(x)(1+\hat\omega(x))$
for all $x\in H$.
Since $f\in C^1(H)$, the remaining condition of \ref{2.1} is satisfied.
Finally, the conditions \ref{2.2} and \ref{2.3} hold since
$|f|\le |f-g|+|g|\le\tfrac15\min(\omega,\phi)+\frac12\min(\omega,\rho_{G_1}^2)
\le\omega\1_{\{\phi>0\}}$,
and $\|f'(x)-\psi(x)e\|\le 6\min(\tfrac1k,\phi(x))\le\sigma\1_{\{\phi>0\}}(x)$ and $f'=g'=0$ and $\psi=0$ outside $G_0=\{\omega>0\}$.
\end{proof}

In a rather straightforward way, we will use
Lemma~\ref{2} recursively to obtain the main tool for
our construction of a function non-differentiable at points of a given set $E$.

\begin{lem}\label{2x}
Suppose $E\subset H_0\subset \R^n$, $H_0$~is open,
$f_0\in\Lip(\R^n)\cap C^1(H_0)$ and
$\omega_0\in C(\R^n,[0,\infty))\cap C(H_0,(0,\infty))$.
Suppose further that for $k\ge1$ we are given
vectors $e_k\in\B$,
functions $\phi_k\in C(\R^n,[0,1])$
and $\sigma_k>0$ such that
$w_{e_k,\taux(\sigma_k)}(E\cap \{\phi_k>0\})=0$.
Then for each $j\ge1$ there are sets $H_j\subset \R^n$ and functions
$f_j,\omega_j, \psi_j:\R^n\to\R$ such that
\begin{enumerate}
\item\label{2x.F}
$H_j$ is open, $E\subset H_j\subset H_{j-1}$ and
$f_j\in\Lip(\R^n)\cap C^1(H_j)$;
\item\label{2x.1}
$\omega_j\in C(\R^n,[0,\infty))\cap C(H_j,(0,\infty))$
and $\omega_j\le\tfrac12\min(1,\omega_{j-1},\rho_{H_{j}}^2)$;
\item\label{2x.0}
$|f_j-f_{j-1}|\le \omega_{j-1}$ and $f_j(x)=f_{j-1}(x)$ when $\phi_j(x)=0$;
\item\label{2x.5}
if $h:\R^n\to\R$ and $|h-f_j|\le 2\omega_j$ then for every
$x\in H_j$ one may find $0<r<\omega_{j-1}(x)$
such that $\sup_{\|y\|\le r}|h(x+y)-h(x)-\spr{f_j'(x)}{y}|\le\sigma_j r$;
\item\label{2x.7}
$\psi_j:\R^n\to[0,1]$,
$0\le\psi_j\le\phi_j\1_{H_{j-1}}$ and $\psi_j=\phi_j$ on $H_j$;
\item\label{2x.8}
$\|f_j'(x)-f_{j-1}'(x)-\psi_j(x)e_j\|\le \sigma_j\1_{\{\phi_j>0\}}(x)$
for every $x\in E$;
\item\label{2x.9}
$\|f_j'(x)-\vv\|\le \|f_0'(x)+\sum_{i=1}^j\psi_i(x)e_i-\vv\|
+\sum_{i=1}^j \sigma_i\1_{\{\phi_i>0\}}(x)$ for any $\vv\in\R^n$ and a.e.~$x\in\R^n$.
\end{enumerate}
\end{lem}

\begin{proof}
Replacing $\omega_0$ by
$\frac12\min(1,\omega_0,\rho_{H_0}^2)$ if necessary,
we may and will assume that $\omega_0\le\tfrac12 \min(1,\rho_{H_0}^2)$ and observe that then $H_0=\{\omega_0>0\}$.
Assume $j\ge 1$ and an open set $H_{j-1}\supset E$, a function $f_{j-1}\in \Lip(\R^n)\cap C^1(H_{j-1})$, and a function $\omega_{j-1}\in C(\R^n,[0,\infty))\cap C(E,(0,\infty))$ such that
$H_{j-1}=\{\omega_{j-1}>0\}$,
have been already defined;
this is certainly the case for $j=1$.
We will now explain how to construct functions $f_j,\omega_j,\psi_j$ and sets $H_j$ such that conditions \ref{2x.F}--\ref{2x.8} of the present lemma are satisfied. Notice that once we construct these objects, we have an open set $H_{j}\supset E$ and a function $f_{j}\in \Lip(\R^n)\cap C^1(H_{j})$ from \ref{2x.F}, and a function $\omega_{j}\in C(\R^n,[0,\infty))\cap C(E,(0,\infty))$ satisfying $\omega_{j}\le\min(1,\rho_{H_{j}}^2)$ for all $x\in\R^n$ from \ref{2x.1}. This will allow us recursively to construct all required objects so that \ref{2x.F}--\ref{2x.8} hold, and then we will finish the proof by showing that \ref{2x.9} holds as well.

By Lemma~\ref{2} find $g_j,\psi_j$ and
$H_j\subset \R^n$
such that
\begin{enumerate}[label=(\alph*)]
\item\label{2xp.1}
$H_j$ is open,
$E\subset H_j\subset H_{j-1}$ and
$g_j\in \Lips(\R^n)\cap C^1(H_j)$;
\item\label{2xp.2}
$|g_j(x)|\le\omega_{j-1}(x)\|e_j\|$ for all $x\in\R^n$ and $g_j(x)=0$
when $\phi_j(x)=0$;
\item\label{2xp.3}
$\|g_j'(x)-\psi_j(x)e_j\|\le\sigma_j\1_{H_{j-1}}(x)\1_{\{\phi_j>0\}}(x)\|e_j\|$
for almost all~$x\in\R^n$;
\item\label{2xp.4}
$0\le\psi_j(x)\le\phi_j(x)\1_{H_{j-1}}(x)$ for $x\in\R^n$ and
$\psi_j(x)=\phi_j(x)$ for $x\in H_j$.
\end{enumerate}
Here we used that
$H_{j-1}=\{\omega_{j-1}>0\}$ to obtain
conditions~\ref{2xp.1}--\ref{2xp.4} directly from
conditions~\ref{2.1}--\ref{2.4} of Lemma~\ref{2}.

Let $f_j:=f_{j-1}+g_j$, then \ref{2xp.1} and \ref{2xp.2} imply \ref{2x.F} and
\ref{2x.0}, respectively.
By Lemma~\ref{3x} we may find
$\xi_j\in C(\R^n,[0,\infty))\cap C(H_j,(0,\infty))$
having the property that
whenever $x\in H_j$ and $h:\R^n\to\R$ satisfies
$|h-f_j|\le\xi_j$, there is $0<r<\omega_{j-1}(x)$
such that $|h(x+y)-h(x)-\spr{f_j'(x)}{y}|\le\eta_j r$
whenever $\|y\|\le r$.
Letting $\omega_j:=\frac12\min(\omega_j,\xi_j,\rho_{H_j}^2)$,
we have \ref{2x.1} and \ref{2x.5}.
Clearly, \ref{2x.7} is the same as \ref{2xp.4}, and \ref{2xp.3}
implies that
\begin{equation}\label{2x.17}
\|f_j'(x)-f_{j-1}'(x)-\psi_j(x)e_j\|\le \sigma_j\1_{\{\phi_j>0\}}(x)
\end{equation}
for almost every $x\in\R^n$. From this,
since $f_j',f_{j-1}'$ and $\psi_j=\phi_j$ are continuous
on the open set $H_j\supset E$, we have \ref{2x.8}.

By the recursive use of the above construction we have defined
$H_j,f_j,\omega_j$ and $\psi_j$ such that \ref{2x.F}--\ref{2x.8} hold.
The last required statement \ref{2x.9}
follows by using \eqref{2x.17} to estimate,
for almost every $x\in\R^n$,
\begin{align*}
\|f_j'(x)-\vv\|&\le \|f_0'(x)+\sum_{i=1}^j\psi_i(x)e_i-\vv\|
+ \sum_{i=1}^j\|f_i'(x)-f_{i-1}'(x)-\psi_i(x)e_i\|\\
&\le
\|f_0'(x)+\sum_{i=1}^j\psi_i(x)e_i-\vv\| +\sum_{i=1}^j \sigma_i\1_{\{\phi_i>0\}}(x).
\qedhere
\end{align*}
\end{proof}

We will use Lemma~\ref{2x} to prove the
two key results, Theorem~\ref{4} and Theorem~\ref{9}.
To prove the former, we will
choose the objects required in Lemma~\ref{2x}
using the following combination
of suitable partitions of unity.

\begin{lem}\label{3}
Suppose $E\subset\R^n$ is cone unrectifiable and $\eps>0$.
Then there exist sequences of positive numbers $\sigma_l > 0$,
vectors $e_l\in\B$ and continuous functions $\phi_l:\R^n\to[0,1]$, such that
\begin{enumerate}
\item\label{3.1}
$\sum_{l\ge1}\sigma_l\1_{\spt(\phi_l)}\le\eps$;
\item\label{3.3}
$w_{e_l,\taux(\sigma_l)}(E\cap\{\phi_l>0\})=0$ for each $l\ge1$;
\item\label{3.4}
if $x\in E$, $e\in\NN(E,x)$ and $\|e\|\le 1$,
then for every $\eta>0$ there are arbitrarily large $l$
such that $\sigma_l<\eta$, $\|e-e_l\|<\eta$ and $\phi_l(x)=1$.
\end{enumerate}
\end{lem}

\begin{proof}
For $x\in E$, $e\in\NN(E,x)$ and any $\sigma>0$ there exists, by definition of the cone unrectifiable set, a radius
$\delta(x,e,\sigma)>0$ such that
$w_{e,\taux(\sigma)}(E\cap B_{x,e,\sigma})=0$, where $B_{x,e,\sigma}=B(x,\delta(x,e,\sigma))$.

We may suppose $\eps=1/p$ for some $p\in\N$ (so that $1/\eps$ is a positive integer).
For each $i\ge1$ we let $\eps_i:=2^{-i}\eps$ and
$\tau_i:=3^{-n}\eps_{i}^{n+1}(n+1)^{-1}$.
For each pair of $i\ge1$ and $j=1,\dots, 3^n\eps_i^{-n}$ choose $e_{i,j}\in\B$
such that $\B\subset\bigcup_j B(e_{i,j},\eps_i)$ for every fixed $i\ge1$.
Let
\[E_{i,j}:=\{x\in E:(\exists e\in\NN(E,x))\|e-e_{i,j}\|<\eps_i\},\]
so that of course $\bigcup_j E_{i_0,j}=E$ for each fixed $i_0\ge1$.
For each pair $(i_0,j_0)$ find a partition of unity $\{\phi_{i_0,j_0,k}:k\ge1\}$
of order~$n$ subordinated to
\[\{B_{y,u,\sigma}:
y\in E_{i_0,j_0},\,u\in \NN(E,y),\,\|u-e_{i_0,j_0}\|<\eps_{i_0}, \sigma=\tau_{i_0}\}.\]

Order the triples $(i,j,k)$
into a single sequence
$(i(l),j(l),k(l))$, and let
$\phi_l:=\min\bigl(1,(n+1)\phi_{i(l),j(l),k(l)}\bigr)$
and $\sigma_l:=\tau_{i(l)}$. Also, observing that
$\spt(\phi_l)=\spt(\phi_{i(l),j(l),k(l)})$, find
$y_l\in E_{i(l),j(l)}$ and $e_l\in\NN(E,y_l)$
such that $\spt(\phi_l)\subset B_{y_l,e_l,\sigma_l}$.
Notice for future reference that $\|e_l-e_{i(l),j(l)}\|<\eps_{i(l)}$.

We show that the Lemma holds with
the $\sigma_l$, $e_l$ and $\phi_l$ defined above.

To prove \ref{3.1}, observe that
for each fixed $i_0\ge1$ and $x_0\in\R^n$ there are
at most
$3^n\eps_{i_0}^{-n}(n+1)$ pairs $(j,k)$ for which $x_0\in\spt(\phi_{i_0,j,k})$.
Notice also that $\sigma_l$ is constant and
equal $\tau_{i_0}$ over all $l$ with the same value of $i(l)=i_0$.
Hence
\[\sum_l\sigma_l\1_{\spt(\phi_l)}(x_0)
\le\sum_i 3^n\eps_i^{-n}(n+1)\tau_{i} \le
\sum_i \eps_i\le\eps.\]

The statement \ref{3.3} is immediate from
$w_{e_l,\taux(\sigma_l)}(E\cap B_{y_l,e_l,\sigma_l})=0$ and the inclusion
$\spt(\phi_l)\subset B_{y_l,e_l,\sigma_l}$.

Finally, suppose $x\in E$, $e\in\NN(E,x)$, $\|e\|\le 1$, $\eta>0$ and $l_0\in\N$.
Let $i_0>\max\{i(l); l\le l_0\}$ be such that $\eps_{i_0}<\eta/2$.
For any $i>i_0$ there is $j$ such that $\|e-e_{i,j}\|<\eps_i<\eps_{i_0}<\eta/2$.
Then $x\in E_{i,j}$ and since the partition of unity $\{\phi_{i,j,k}:k\ge1\}$
is of order~$n$, there is $k$ such that
$\phi_{i,j,k}(x)\ge 1/(n+1)$. This implies $\phi_l(x)=1$ for $l$ satisfying $(i,j,k)=(i(l),j(l),k(l))$.
Then $l>l_0$ and $\sigma_l=\tau_i<\eps_i$, so
$\|e-e_l\|\le\|e-e_{i,j}\|+\|e_{i,j}-e_l\|<2\eps_i<\eta$, so \ref{3.4} holds as well.
\end{proof}

Our second use of Lemma~\ref{2x}, to prove Theorem~\ref{9},
will be more straightforward: we use it
to construct functions that will approximate the
required function.

\begin{lem}\label{U}
Suppose $E\subset H\subset \R^n$,
$E$ is uniformly purely unrectifiable, $H$~is open,
$\omega\in C(\R^n,[0,\infty))\cap C(H,(0,\infty))$
and $f\in\Lip(\R^n)\cap C^1(H)$.
Then for every
$e\in\R^n$ and
$\eta>0$
there are $g,\xi:\R^n\to\R$ and an open set $U\subset\R^n$ such that
\begin{enumerate}
\item\label{pU.1}
$E\subset U\subset H$,
$\xi\in C(\R^n,[0,\infty))\cap C(U,(0,\infty))$
and $\xi\le\tfrac12\omega$;
\item\label{pU.2}
$|g-f|\le \omega$, $\Lip(g)\le\max(\Lip(f),\|e\|)+\eta$
and $g \in C^1(U)$;
\item\label{pU.4}
if $x\in E$ and a function $h:\R^n\to\R$ satisfies
$|h-g|\le 2\xi$, there is $0<r<\omega(x)$
such that $\sup_{\|y\|\le r}|h(x+y)-h(x)-\spr{e}{y}|\le\eta r$.
\end{enumerate}
\end{lem}

\begin{proof}
Let $\sigma=\eta/8(n+1)$. Since $f\in C^1(H)$ and $E\subset H$,
for each $x\in E$ there is $\delta_x>0$ such that
$\|f'(y)-f'(z)\|<\tfrac14\eta$
for $y,z\in B_x:=B(x,\delta_x)$.
Find a partition of unity $\{\gamma_{k}:k\ge1\}$
of order~$n$ subordinated to $\{B_x: x\in E\}$
and choose
$x_k\in E$ such that $\spt(\gamma_k)\subset B_{x_k}$.

Set $H_0=H$, $\omega_0=\tfrac12\omega$, $f_0=f$, $\sigma_k=\sigma$,
$e_{2k-1}=-f'(x_k)\in\NN(E,x_k)$, $e_{2k}=e\in\NN(E,x_k)$, and $\phi_{2k-1}=\phi_{2k}=\gamma_k$. Since $E$ is uniformly purely unrectifiable, the hypothesis of Lemma~\ref{2x} is satisfied, and so
find $f_k$, $\omega_k$, $H_k$ and $\psi_k$, $k\ge 1$, such that the statements \ref{2x.F}--\ref{2x.9} of Lemma~\ref{2x} hold (we leave out \ref{2x.5} and \ref{2x.8} as we do not use them here):
\begin{enumerate}[label=(\alph*)]
\item\label{UP.0}
$H_k$ is open, $E\subset H_k\subset H_{k-1}$ and
$f_k\in\Lip(\R^n)\cap C^1(H_k)$;
\item\label{UP.2}
$\omega_k\in C(\R^n,[0,\infty))\cap C(H_k,(0,\infty))$
and $\omega_k\le\tfrac12\min(1,\omega_{k-1},\rho_{H_{k}}^2)$;
\item\label{UP.1}
$|f_k-f_{k-1}|\le \omega_{k-1}$ and $f_k(x)=f_{k-1}(x)$ when $\phi_k(x)=0$;
\item\label{UP.7}
$\psi_k:\R^n\to[0,1]$,
$0\le\psi_k\le\phi_k\1_{H_{k-1}}$ and $\psi_k=\phi_k$ on $H_k$;
\item\label{UP.9}
$\|f_k'(x)-\vv\|\le \|f'(x)+\sum_{i=1}^k\psi_i(x)e_i-\vv\|
+\sum_{i=1}^k \sigma\1_{\{\phi_i>0\}}(x)$ for all $\vv\in\R^n$ and a.e.~$x\in\R^n$.
\end{enumerate}

By \ref{UP.2} and \ref{UP.1},
the sequence of Lipschitz functions $(f_k)$ converges to a function $g:\R^n\to\R$
and $|g-f|\le\omega$. For every $x$ at which $f'(x)$ exists write
\begin{equation}\label{UPE.1}
f'(x)+\sum_{i=1}^{2k} \psi_i(x)e_i =a f'(x)+b e +v,
\end{equation}
where
$a=1-\sum_{i=1}^k \psi_{2i-1}(x)$, $b=\sum_{i=1}^k \psi_{2i}(x)$,
$v=\sum_{i=1}^k \psi_{2i-1}(x)(f'(x)-f'(x_i))$.
Using $\sum_i\gamma_i\le 1$ as it is a partition of unity, and \ref{UP.7} to get
\begin{equation}\label{UPE.4}
0\le\psi_{2i}\le\phi_{2i}\1_{H_{2i-1}}=\phi_{2i-1}\1_{H_{2i-1}}\le\psi_{2i-1}\le\phi_{2i-1}=\gamma_i,
\end{equation}
we see that $a,b\ge 0$, $a+b=1+\sum_{i=1}^k (\psi_{2i}(x)-\psi_{2i-1})(x)\le 1$,
and
$\|v\|\le\sum_{i: x\in\spt(\gamma_i)}\gamma_i(x)\|f'(x)-f'(x_i)\|$.
Recall that $\spt(\gamma_i)\subset B_{x_i}$,
and by the definition of the ball $B_{x_i}$
we have $\|f'(x)-f'(x_i)\|<\tfrac14\eta$ for $x\in B_{x_i}$,
hence $\|v\|<\tfrac14\eta$. Thus we conclude from \eqref{UPE.1} that for almost all $x\in\R^n$ and all $k\ge1$
\begin{equation}\label{UPE.2}
\Bigl\|f'(x)+\sum_{i=1}^{2k} \psi_i(x)e_i\Bigr\|
\le\max(\Lip(f),\|e\|)+\eta/4.
\end{equation}
Since for every $x$ there are at most $2(n+1)$ values of $i$
with $\phi_i(x)\ne 0$, we see that
$\sum_{i=1}^{2k} \sigma\1_{\{\phi_i>0\}}(x)\le 2(n+1)\sigma=\tfrac14\eta$
for any $k\ge1$, and infer from
\ref{UP.9} with $\vv=0$ and \eqref{UPE.2} that for a.e.~$x$,
\[\|f_{2k}'(x)\|\le \|f'(x)+\sum_{i=1}^{2k}\psi_i(x)e_i\|
+\sum_{i=1}^{2k} \sigma\1_{\{\phi_i>0\}}(x)\le \max(\Lip(f),\|e\|)+\tfrac12\eta.\]
Since, by \ref{UP.0}, $f_{2k}$ is Lipschitz, we conclude
$\Lip(f_{2k})< \max(\Lip(f),\|e\|)+\eta$ for each $k$,
and so \ref{pU.2} holds.

For each $x\in E$ there is a neighbourhood where all but a finite number
of the functions $\phi_k$'s are zero,
so we can find $r_x>0$ and $k_x\in\N$ such that
$B(x,r_x)\cap\spt{\phi_k}=\emptyset$ for $k\ge k_x$.
Let $U_x:=B(x,r_x)\cap H_{k_x}$, where $H_{k_x}\supset E\ni x$ is defined in \ref{UP.0},
and define an open set $U:=\bigcup_{x\in E} U_x$.
As $x\in U_x\subset H_{k_x}\subset H_0=H$ for any $x\in E$, we conclude that $E\subset U\subset H$, this verifies the first two
statements of \ref{pU.1}.
By \ref{UP.1}, $g=f_{k}$ on $B(x,r_x)\supset U_x$ for every $k\ge k_x$;
hence $g\in C^1(U_x)$ by \ref{UP.0} as $U_x\subset H_{k_x}$,
and so $g\in C^1(U)$.
Thus Lemma~\ref{3x} applied to $U,g,\omega$ and $\frac12\eta$ provides a continuous function $\xi:\R^n\to[0,\infty)$
such that \ref{pU.1} holds and for every
$x\in E\subset U$ and $h:\R^n\to\R$ satisfying
$|h-g|\le 2\xi$, there is $0<r<\omega(x)$
such that
\begin{equation}\label{UPE.3}
\sup_{\|y\|\le r}|h(x+y)-h(x)-\spr{g'(x)}{y}|\le\tfrac12\eta r.
\end{equation}
Observe now that for $x\in E$ we have $x\in H_i$ for any $i\ge1$, hence $\psi_i(x)=\phi_i(x)$ for any $i\ge1$ by \ref{UP.7}. Together with definition of $k_x$ this implies that $\sum_{i=1}^k\psi_{2i-1}(x)=\sum_{i=1}^k\phi_{2i-1}(x)=\sum_{i=1}^k\gamma_i(x)=\sum_{i\ge1}\gamma_i(x)=1$ for any $k\ge k_x$, hence for such $k$ the constants
$a,b$ from~\eqref{UPE.1} satisfy
$a=0$ and, similarly, $b=1$.
Using equation~\eqref{UPE.1} and
recalling that $\|v\|\le\tfrac14\eta$,
we get
$\|f'(x)+\sum_{i=1}^{2k} \psi_i(x)e_i -e\|=\|v\|\le\tfrac14\eta$ for any $k\ge k_x$.
With $k=k_x$ we have $g=f_{2k}$ on $U_x$, hence using \ref{UP.9} with $\vv=e$ it follows
\[\|g'(x)-e\|=\|f_{2k}'(x)-e\|
\le \|f'(x)+\sum_{i=1}^{2k} \psi_i(x)e_i -e\|+ \sigma\sum_{i=1}^{2k}\1_{\{\phi_i>0\}}(x)
\le\tfrac12\eta,\]
and by combining this with \eqref{UPE.3}, we obtain \ref{pU.4}.
\end{proof}

\section{Proofs of main results}\label{proofs}

\begin{proof}[Proof of Theorem \ref{4}]
Recall that we are given a cone unrectifiable set $E\subset \R^n$.
We are also given $\eps>0$ and a continuous function $\omega\ge 0$
such that $E\subset\{x:\omega(x)>0\}$; if $\omega$ is not given, we set
$\omega=1$ everywhere on $\R^n$.

We begin by finding numbers $\sigma_k >0$,
vectors $e_k\in\B$ and continuous functions $\phi_k:\R^n\to[0,1]$, $k=1,2,\dots$, such that
\begin{enumerate}[label=(\Alph*)]
\item\label{9p.1}
$\sum_k\sigma_k\1_{\spt(\phi_k)}\le\eps$;
\item\label{9p.3}
$w_{e_k,\taux(\sigma_k)}(E\cap\{\phi_k>0\})=0$;
\item\label{9p.4}
if $x\in E$, $e\in\NN(E,x)$ and $\|e\|\le 1$,
then for every $\eta>0$ there are arbitrarily large $k$
such that $\sigma_{2k-1}<\eta$, $\|e-e_{2k-1}\|<\eta$ and $\phi_{2k-1}(x)=1$;
\item\label{9p.5}
for every $k\ge 1$, $\phi_{2k}=\phi_{2k-1}$ and $e_{2k}=-e_{2k-1}$.
\end{enumerate}
For this, it suffices to take $\hat\sigma_l$,
$\hat e_l$ and $\hat\phi_l$ from Lemma~\ref{3}
with $\eps$ replaced by $\eps/2$ and let
$\sigma_{2l-1}=\sigma_{2l}:=\hat\sigma_l$,
$\phi_{2l-1}=\phi_{2l}:=\hat\phi_l$,
$e_{2l-1}:=\hat e_l$ and $e_{2l}:=-\hat e_l$.

We set $f_0:=0$, $H_0:=\{\omega>0\}$,
$\omega_0:=\tfrac12\min(1,\omega,\rho_{H_{0}}^2)$
and use Lemma~\ref{2x} to find
$f_j,\omega_j, H_j,\psi_j$, $j=1,2,\dots$ such that

\begin{enumerate}[resume*]
\item\label{9px.F}
$H_j$ is open, $E\subset H_j\subset H_{j-1}$ and
$f_j\in\Lip(\R^n)\cap C^1(H_j)$;
\item\label{9px.4}
$\omega_j\in C(\R^n,[0,\infty))\cap C(H_j,(0,\infty))$
and $\omega_j\le\tfrac12\min(1,\omega_{j-1},\rho_{H_{j}}^2)$;
\item\label{9px.0}
$|f_j-f_{j-1}|\le \omega_{j-1}$ and $f_j(x)=f_{j-1}(x)$ when $\phi_j(x)=0$;
\item\label{9px.5}
if $h:\R^n\to\R$ and $|h-f_j|\le 2\omega_j$ then for every
$x\in H_j$ one may find $0<r<\omega_{j-1}(x)$
such that $\sup_{\|y\|\le r}|h(x+y)-h(x)-\spr{f_j'(x)}{y}|\le\sigma_j r$;
\item\label{9px.7}
$\psi_j:\R^n\to[0,1]$,
$0\le\psi_j\le\phi_j\1_{H_{j-1}}$ and $\psi_j=\phi_j$ on $H_j$;
\item\label{9px.6}
$\|f_j'(x)-f_{j-1}'(x)-\psi_j(x)e_j\|\le\sigma_j\1_{\{\phi_j>0\}}(x)$
for every $x\in E$;
\item\label{2px.9}
$\|f_j'(x)-\vv\|\le \|f_0'(x)+\sum_{i=1}^j\psi_i(x)e_i-\vv\|
+\sum_{i=1}^j \sigma_i\1_{\{\phi_i>0\}}(x)$ for all $\vv\in\R^n$ and a.e.~$x\in\R^n$.
\end{enumerate}

Notice that \ref{9px.4} implies $\omega_j\le 2^{i-j}\omega_j$ for $j\ge i$, and so also $\omega_j\le 2^{-j}$.
Consequently, by \ref{9px.0}, $f_j$ converge uniformly to a function
$f:\R^n\to\R$ and
$|f-f_j|\le \sum_{i=j}^\infty \omega_i \le 2\omega_j$.
We show that $f$ has the required properties.

Notice that \ref{9px.7} and \ref{9p.5} imply
that
\[\psi_{2i-1}(x) e_{2i-1}+\psi_{2i}(x)e_{2i}=-(\psi_{2i-1}(x)-\psi_{2i}(x))\1_{H_{2i-2}\setminus H_{2i}}(x)e_{2i},\]
and this vector has norm at most $\1_{H_{2i-2}\setminus H_{2i}}(x)$, as condition~\ref{9px.7} implies $0\le\psi_{2i}\le\phi_{2i}\1_{H_{2i-1}}=\phi_{2i-1}\1_{H_{2i-1}}\le\psi_{2i-1}\le\phi_{2i-1}\le1$ (cf.~\eqref{UPE.4}). Hence
\ref{2px.9} with $\vv=0$ and \ref{9p.1} give
\begin{align*}
\|f_{2k}'(x)\|
&=\Bigl\|\sum_{i=1}^{k} (\psi_{2i}(x) e_{2i}+\psi_{2i-1}(x)e_{2i-1})\Bigr\|
+ \sum_{i=1}^{2k} \sigma_i\1_{\{\phi_i>0\}}(x)\\
&\le \sum_{i=1}^{k} \1_{H_{2i-2}\setminus H_{2i}}(x)
+\sum_{i=1}^{2k}\sigma_i\1_{\{\phi_i>0\}}(x)\le 1+\eps
\end{align*}
for almost every $x$.
Since \ref{9px.F} shows that $f_{2k}$ is Lipschitz, $\Lip(f_{2k})\le 1+\eps$,
and we conclude that $\Lip(f)\le 1+\eps$.

For every $i\ge1$ and $x\in E\subset H_{2i}\subset H_{2i-1}$, \ref{9px.7}, \ref{9p.5}
and \ref{9px.6} imply
\begin{align*}
\|&f_{2i}'(x)-f_{2i-2}'(x)\|\\
&\;\;=\|(f_{2i}'(x)-f_{2i-1}'(x)-\phi_{2i}(x)e_{2i})
+(f_{2i-1}'(x)-f_{2i-2}'(x)-\phi_{2i-1}(x)e_{2i-1})\|\\
&\;\;\le \sigma_{2i}\1_{\{\phi_{2i}>0\}}(x)+\sigma_{2i-1}\1_{\{\phi_{2i-1}>0\}}(x).
\end{align*}
Since $\sum_j\sigma_j\1_{\{\phi_j>0\}}(x)\le \eps$ by \ref{9p.1}, the restrictions of
$f_{2k}'$ to $E$ converge pointwise to a function
$u:E\to\R^n$ and $\|u(x)\|\le\eps$ for $x\in E$.

Suppose $x\in E$, $e\in\NN(E,x)$, $\|e\|\le 1$ and $\eta>0$.
By \ref{9p.4}
there is $k$ such that $2^{-2k}<\eta$,
$\|f_{2k}'(x)-u(x)\|<\tfrac14\eta$,
$\|e-e_{2k+1}\|<\tfrac14\eta$, $\sigma_{2k+1}<\tfrac14\eta$ and $\phi_{2k+1}(x)=1$. Since $x\in E\subset H_{2k+1}$, the latter immediately implies $\psi_{2k+1}(x)=1$ by \ref{9px.7}.
Since
$|f-f_{2k+1}|\le 2\omega_{2k+1}$
and \ref{9px.6} gives $\|f_{2k+1}'(x)-(f_{2k}'(x)+e_{2k+1})\|\le\sigma_{2k+1}$, we conclude that
\ref{9px.5} provides $0<r <\omega_{2k}(x)\le 2^{-2k}\omega_0<\eta$ such that
for every $\|y\|\le r$,
\begin{align*}
|f(x+&y)-f(x)-\spr{u(x)+e}{y}|\\
&\le
|f(x+y)-f(x)-\spr{f_{2k+1}'(x)}{y}| + \|f_{2k+1}'(x)-(f_{2k}'(x)+e_{2k+1})\| \|y\|\\
&\quad + \|f_{2k}'(x)-u(x)\| \|y\| +\|e_{2k+1}-e\|\|y\|\\
&<(\sigma_{2k+1} +\sigma_{2k+1}+\eta/4+\eta/4) r<
\eta r.
\end{align*}
Since $\eta>0$ may be arbitrarily small,
\begin{equation}\label{E1a}
\adjustlimits\liminf_{r\searrow 0}\sup_{\|y\|\le r}
\frac{|f(x+y)-f(x)-\spr{e+u(x)}y|}{r} =0,
\end{equation}
which is the main statement we wished to prove.
The estimate of the lower and upper derivatives is an immediate
consequence: if $e\in\NN(E,x)$ and $\|e\|\le 1$, we use
\eqref{E1a} for $e$ and $-e$ to infer
\[{D^+}f(x;y)-{D_+}f(x;y)
\ge
\spr{e+u(x)}y - \spr{-e+u(x)}y = 2\spr ey.\qedhere\]
\end{proof}

\begin{proof}[Proof of Corollary \ref{CAM}]
We are given $E=\bigcup_{k\ge1} E_k\subset\R^n$
where $E_k$ are disjoint cone unrectifiable $F_\sigma$ sets, and
$\NN_x=\NN(E_k,x)\cap\B$ for $x\in E_k$.

Write $E_k=\bigcup_{j\ge1} H_{k,j}$ where $H_{k,j}$ are closed cone unrectifiable sets, and let
$F_{k,j}:=\bigcup_{i<j} H_{k,i}$ and $E_{k,j}:=H_{k,j}\setminus F_{k,j}$,
so that $E_{k,j}$ are pairwise disjoint over all $(k,j)$.
Let $c_{k,j}:=2^{-k-j}$ and
$\omega_{k,j}(x):= c_{k,j}\min(1,\dist^2(x,F_{k,j}))$.
By Theorem \ref{4} there are Lipschitz functions $f_{k,j}:\R^n\to\R$ such that
$\Lip(f_{k,j}) < 2$, $|f_{k,j}|\le \omega_{k,j}$ and
\begin{align*}
{D}^+f_{k,j}(x;y)-{D}_+f_{k,j}(x;y)&\ge 2\sup\{\spr ey : e\in\NN(E_{k,j},x),\,\|e\|\le1\}\\
&\ge2\sup_{e\in\NN_x} \spr ey
\end{align*}
for $x\in H_{k,j}$ and $y\in\R^n$;
the last inequality follows from
$\NN_x\subset\NN(E_{k,j},x)$.

Apply Lemma~\ref{AL} to $\omega=\omega_{k,j+1}$, $H=\{\omega_{k,j+1}>0\}$, $g=f_{k,j}$, $\Phi=0$ and $\xi=2$ to find Lipschitz functions $g_{k,j}:\R^n\to\R$ such that
$g_{k,j}\in C^1\{\omega_{k,j+1}>0\}$,
$|g_{k,j}-f_{k,j}|\le \omega_{k,j+1}$ and $\Lip(g_{k,j})\le 3$.
We observe that $g_{k,j}$ is differentiable at every $x\notin H_{k,j}$.
Indeed, for such an $x$, if $\omega_{k,j}(x)=0$, i.e.\ $x\in F_{k,j}\subset F_{k,j+1}$, then
$g_{k,j}(x)=f_{k,j}(x)=0$ as $\omega_{k,j}(x)=\omega_{k,j+1}(x)=0$, and $|g_{k,j}(y)|\le2c_{k,j}\|y-x\|^2\le\|y-x\|^2$, using upper estimates for $|g_{k,j}-f_{k,j}|$ and $|f_{k,j}|$, and $x\in F_{k,j}\subset F_{k,j+1}$; hence $g_{k,j}'(x)=0$. If, however, $x\notin H_{k,j}$ and $\omega_{k,j}(x)>0$, then $x\notin E_{k,j}\cup F_{k,j}$, hence
$\omega_{k,j+1}(x)>0$ and so it follows that $g_{k,j}$ is $C^1$ on a neighbourhood of~$x$.
We also observe that for every $x\in H_{k,j}$ and $y\in\R^n$, we have $x\in F_{k,j+1}$, and therefore
$|g_{k,j}(y)-f_{k,j}(y)|\le c_{k,j+1}\|y-x\|^2$ and hence $g_{k,j}(x)=f_{k,j}(x)$ and
\begin{equation}\label{E1b}
{D}^+g_{k,j}(x;y)-{D}_+g_{k,j}(x;y)
= {D}^+f_{k,j}(x;y)-{D}_+f_{k,j}(x;y)\ge 2\sup_{e\in \NN_x}\spr ey.
\end{equation}
Summarising, $g_{k,j}$ is differentiable at every $x\not\in H_{k,j}$ and is not differentiable at any $x\in H_{k,j}$, moreover, it satisfies \eqref{E1b} at such points $x$.

We let
$\displaystyle f:=\sum_{(s,t)} c_{s,t}g_{s,t}$ and
$\displaystyle h_{k,j}:=\sum_{(s,t)\ne(k,j)} c_{s,t}g_{s,t}$. Since
for any $(s,t)$, if $x\notin H_{s,t}$, then the function $g_{s,t}$
is differentiable at $x$, and
since we have $\sum_{s,t} \Lip( c_{s,t}g_{s,t})<\infty$, we infer that
$f$ is differentiable at any $\displaystyle x\notin\bigcup_{(s,t)}H_{s,t}=E$ and
$h_{k,j}$ is differentiable at any $x\in H_{k,j}\cup(\R^n\setminus E)$.

Let $x\in E_k$ and find $j$ such that $x\in H_{k,j}$.
Then for every $y\in\R^n$,
${D}^+g_{k,j}(x;y)-{D}_+g_{k,j}(x;y)\ge 2\sup_{e\in \NN_x}\spr ey$ by \eqref{E1b},
and so, since $f=c_{k,j}g_{k,j} + h_{k,j}$ and $h_{k,j}$ is differentiable at $x$,
we conclude that
\[{D}^+f(x;y)-{D}_+f(x;y)\ge 2c_{k,j} \sup_{e\in \NN_x}\spr ey.\qedhere\]
\end{proof}

\begin{proof}[Proof of Corollary \ref{CAMx}]
We are given a set $E\subset\R^n$
that is a countable union of (not necessarily disjoint) cone unrectifiable $F_\sigma$ sets.
Since each of these $F_\sigma$ sets is a countable union of closed
cone unrectifiable sets, we can write $E=\bigcup_{k=1}^\infty F_k$ where
$F_k$ are closed and cone unrectifiable.
Hence $E=\bigcup_{k=1}^\infty E_k$ where $E_k:=F_k\setminus\bigcup_{j<k} F_j$
are disjoint cone unrectifiable $F_\sigma$ sets, and it suffices to
take the function $f$ obtained
from Corollary~\ref{CAM} used with these sets $E_k$.
\end{proof}

\begin{proof}[Proof of Corollary \ref{ACP-AM}]
We are given a Radon measure $\mu$ on $\R^n$ and
a $\mu$-measurable map $T:\R^n\to\bigcup_{m=0}^n G(n,m)$
such that for every unit vector $e$ and $\alpha\in(0,1)$, the set
$\{x: C_{e,\alpha}\cap T(x)=\nolinebreak\{0\}\}$,
where $C_{e,\alpha}:=\{u: |\spr ue|\ge\alpha\|u\|\}$,
is the union of a $\mu$-null set and a set $E$ with $w_{e,\alpha}(E)=0$.
We show that there are cone unrectifiable $F_\sigma$ sets $E_k$
such that $\mu(\R^n\setminus\bigcup_k E_k)=0$ and
$T(x)^\perp \subset\NN(x,E_k)$ for every $x\in E_k$.
Then the function $f$ from Corollary~\ref{CAM} will
have all the required properties.

By Lusin's Theorem, $\mu$-almost all of $\R^n$ is covered by
the union of disjoint closed sets $F_k$ such that
for each $k$, the restriction of $T$ to $F_k$ is continuous.
For every rational $\alpha\in(0,1)$ and $u$ from a countable dense subset $Q$ of the unit sphere in $\R^n$
write $\{x: C_{u,\alpha}\cap T(x)=\nolinebreak\{0\}\}=Z_{u,\alpha}\cup E_{u,\alpha}$,
where
$\mu(Z_{u,\alpha})=\nobreak0$ and $w_{u,\alpha}(E_{u,\alpha})=0$. Letting $E_k$ be
$F_\sigma$ subsets of $F_k\setminus\bigcup_{u,\alpha} Z_{u,\alpha}$ satisfying
$\mu(F_k\setminus E_k)=0$, we just need to show that $T(x)^\perp\subset\NN(x,E_k)$
for $x\in\nobreak E_k$. For this, assume $x\in E_k$, $e\in T(x)^\perp$ and $\eps\in(0,1)$,
and choose $u\in Q$ and rational $\alpha\in(0,1)$ so that
$C_{e,\eps} \subset C_{u,\alpha}$ and $C_{u,\alpha}\cap T(x)=\{0\}$.
By continuity of $T$ on $F_k$, there is $r>0$ such that
$C_{u,\alpha}\cap T(y)=\{0\}$ for every $y\in B(x,r)\cap F_k$.
Hence $B(x,r)\cap E_k\subset E_{u,\alpha}$ and
$w_{e,\eps}(B(x,r)\cap E_k)\le w_{u,\alpha}(E_{u,\alpha})=0$.
\end{proof}

\begin{proof}[Proof of Theorem~\ref{9}]
Let $E$ be the given uniformly purely unrectifiable set.
Pick a sequence $e_k$
dense in the unit ball of $\R^n$ such that $\|e_k\|\le 1-2^{-k}$.

Let $f_0=0$, $H_0=\R^n$, $\omega_0=1$ and $\eta_k=2^{-k-1}$.
When $f_{k-1}$, $H_{k-1}$ and $\omega_{k-1}$ have been defined,
we use Lemma~\ref{U}
to find $f_k$, $H_k$ and $\omega_k:=\xi$ such that
\begin{enumerate}[label=(\alph*)]
\item\label{p9.1}
$E\subset H_k\subset H_{k-1}$,
$\omega_k\in C(\R^n,[0,\infty))\cap C(U,(0,\infty))$
and $\omega_k\le\tfrac12\omega_{k-1}$;
\item\label{p9.2}
$|f_k-f_{k-1}|\le \omega_{k-1}$, $\Lip(f_k)\le \max(\Lip(f_{k-1}),\|e_k\|)+\eta_k$ and $f_k\in C^1(H_k)$;
\item\label{p9.4}
if $x\in E$ and $h:\R^n\to\R$ satisfies
$|h-f_k|\le 2\omega_k$, there is $0<r<\omega_{k-1}(x)$
such that $\sup_{\|y\|\le r}|h(x+y)-h(x)-\spr{e_k}{y}|\le\eta_k r$.
\end{enumerate}

Notice that $\omega_0=1$ and the last inequality in \ref{p9.1}
imply $\omega_j\le 2^{j-k}\omega_k$ and $\omega_k\le 2^{-k}$
for $j\ge k\ge 0$.
From \ref{p9.2} we see by induction that $\Lip(f_k)\le 1-2^{-k-1}$.
Hence the inequality $|f_k-f_{k-1}|\le\omega_{k-1}\le 2^{-k+1}$
implies that $f_k$ converge to some $f:\R^n\to\R$ with $\Lip(f)\le 1$.

Given any $x\in E$, $e\in\R^n$ with $\|e\|\le 1$, and $\eps>0$,
there are arbitrarily large~$k$ such that $\|e_k-e\|<\eps$ and $\eta_k<\eps$.
Inferring from~\ref{p9.2} that
$|f-f_k|\le\sum_{j=k}^\infty \omega_j\le\sum_{j=k}^\infty 2^{j-k}\omega_k \le 2\omega_k$, we use
\ref{p9.4} to find $0<r<\omega_{k-1}(x)\le 2^{-k+1}$
such that
$\sup_{\|y\|\le r}|f(x+y)-f(x)-\spr{e_k}{y}|\le\eta_k r<\eps r$.
Since $\|e_k-e\|<\eps$, we conclude that
$\sup_{\|y\|\le r}|f(x+y)-f(x)-\spr{e}{y}|<2\eps r$.
As $\eps>0$ is arbitrary and $k$ may be arbitrarily large,
\[\adjustlimits\liminf_{r\searrow 0}
\sup_{\|y\|\le r} \frac{|f(x+y)-f(x)-\spr{e}y|}{r} =0,\]
which is the statement \eqref{E2} of the Theorem.
The estimate of upper and lower derivatives follows
by using this with $e=y/\|y\|$ and $e=-y/\|y\|$
to get ${D}^+f(x;y)\ge \|y\|$ and ${D}_+f(x;y)\le -\|y\|$,
respectively.
\end{proof}

\section{Examples}\label{examples}

The argument behind our first
example has already been used many times,
starting with~\cite{p}, to find points of differentiability or
almost differentiability of Lipschitz functions. See, e.g.,
\cite{DH,MD} or \cite[Example 4.7]{AM} for an example showing that
in Corollary~\ref{ACP-AM} the constant $c=c(x)$ cannot be bounded
away from zero.

\begin{exm}\label{e}
There is a compact set $E\subset\R^2$ and a continuous mapping
$x\in E\to e_x\in\{e\in\R^2:\|e\|=1\}$
such that $\NN(E,x)=\{ t e_x: t\in\R\}$ for every $x\in E$
and whenever $f:\nobreak\R^2\to\R$ has $\Lip(f)\le 1$,
there is $x\in E$ such that $\overline{D}f(x,e_x)<1$.
Consequently, in Theorem~\ref{4} we cannot take $\eps=0$.
\end{exm}

\begin{proof}
Let $\phi:\R\to\R$ be a $C^1$ function such that $\phi(-1)=\phi(1)=0$, $\phi'(-1)=\phi'(1)=0$
and $\phi(s)>0$ for $s\ne \pm 1$. Denote $\phi_0=0$ and $\phi_k=\phi/k$, and let
\[E:=\{(s,\phi_k(s)): s\in[-1,1],\,k=0,\pm 1,\pm 2,\dots\}.\]
For $x\in E$, $x=(s,\phi_k(s))$ let $u_x$
and $e_x$ denote the unit vectors in the directions of $(1,\phi_k'(s))$ and $(-\phi_k'(s),1)$, respectively.
Then $e_x\in\NN(E,x)$ and, since $\phi'(-1)=\phi'(1)=0$, the map $x\in E\to e_x$ is continuous.

Suppose $f:\R^2\to\R$ has $\Lip(f)\le 1$.
Consider any
$x\in E\setminus\{(-1,0),(1,0)\}$ such that
$a:=f'(x,u_x)$ exists and $\overline{D}f(x;e_x)=1$. Then
\begin{align*}
\limsup_{t\to 0}\,&\frac{f(x+te_x) - f(x- at u_x)}{t}\\
&\quad=
\limsup_{t\to 0}\frac{f(x+t e_x)-f(x)}{t}
-\lim_{t\to 0} \frac{f(x-{at}u_x)-f(x)}{t}
= 1 + a^2.
\end{align*}
Hence
\[1=\Lip(f)\ge
\limsup_{t\to 0+}\frac{|f(x+te_x) - f(x- at u_x)|}{\|(x+te_x)- (x- at u_x)\|}
=\frac{a^2 +1}{\sqrt{a^2+1}},\]
which gives $f'(x,u_x)=0$.

If $f$ is a function satisfying the conclusion of Theorem~\ref{4} with $\eps=0$,
then for every $k$, $x=(s,\phi_k(s))$
satisfies the above assumptions for a.e.~$s\in(-1,1)$.
Since $f$ is Lipschitz, we infer
that $s\to f(s,\phi_k(s))$ is constant on $[-1,1]$, and hence $f$ is constant on $E$.
Consequently, when $s\in(-1,1)$ and $x=(s,\phi_0(s))$, $e_x=(0,1)$ and so
$\lim_{t\to 0} |(f(x+te_x)-f(x))/t| \le \lim_{t\to 0} \dist(x+te_x,E)/|t|=0$, as
$\dist(x+te_x,E)/|t|\le (k+1)/\bigl(2k(k+1)\bigr)=1/(2k)$
when $|t|$ is between $\phi(x)/(k+1)$ and $\phi(x)/k$.
This contradicts
$\overline{D}f(x;e_x)=1$.
\end{proof}

Our second example is related to Zahorski's description of non-differen\-tiability sets of real-valued functions of a real variable which was already mentioned in the introductory remarks to Corollaries~\ref{CAM} and~\ref{CAMx}. Recall first that
the set of points of non-differentiability of any real-valued function
$f:\R^n\to\R$ is easily seen to be of the type $G_{\delta\sigma}$:
just write it as
\[\bigcup_{\eps>0}\bigcap_{\,e\in\R^n}
\{x: (\exists r>0) (\exists u,v\in B(x,r))
|f(x+u)-f(x+v)-\spr{e}{u-v}|>\eps r\}\]
where $\eps$ runs over positive
rational numbers and $e$ over elements
of a dense countable subset of $\R^n$.
The main argument in Zahorski's~\cite{Z} proof of the converse when $n=1$ (both in the general and in the Lipschitz case) constructs, for a given $G_\delta$ Lebesgue null set $E\subset\R$, a function $f:\R\to\R$ with $\Lip(f)=1$ which is differentiable at every point of $\R\setminus E$ and at every point of $E$ has upper derivative~$1$ and lower derivative~$-1$. (For a more modern treatment of this construction see \cite{FP}.)

While it is not clear what an exact analogy of Zahorski's result for
$n>1$ should be,
one may at least hope that its analogy
holds for uniformly purely unrectifiable sets,
namely that for every uniformly purely unrectifiable
$G_{\delta\sigma}$ set $E\subset \R^n$ there is a Lipschitz function
$f:\R^n\to\R$
such that $E$ is precisely the set of points at which $f$
in non-differentiable in any direction.
We do not know whether this is true or not, but the following
example shows that in this situation the argument
based on uniform discrepancy between upper and lower derivatives
fails in a very strong sense.
Recalling that every uniformly purely unrectifiable set is
contained in a $G_\delta$ uniformly purely unrectifiable set,
the example provides a $G_\delta$ uniformly purely unrectifiable set
such that not only for it, but even for
any bigger $G_\delta$ uniformly purely unrectifiable set
there is no function analogous to the one from Zahorski's
main argument.

\begin{exm}\label{e2}
There is a uniformly purely unrectifiable set $A\subset\R^2$ such that
for any set $E\supset A$ and any $c>0$ there is no Lipschitz function $f:\R^2\to\R$ such that
\begin{enumerate}[label=\rm(\alph*)]
\item\label{e2.1}
${D^+}f(x;y)-{D_+}f(x;y)\ge c\|y\|$ for every $x\in E$ and $y\in\R^2$;
\item\label{e2.2}
$f$ is differentiable at every point $x$ of $\R^2\setminus E$.
\end{enumerate}
\end{exm}

\begin{proof}
By \cite{CPT} there is a universal differentiability set $D\subset\R^2$,
i.e., a set such that
every real-valued Lipschitz function on $\R^2$ has a point
of differentiability belonging to $D$, such that there is
a Lipschitz $h:\R^2\to\R$ for which the set $A$ of points $x\in D$ such that
$h$ is differentiable at $x$, is uniformly purely unrectifiable.
Suppose $E\supset A$ and Lipschitz $f:\R^2\to\R$ satisfy
\ref{e2.1} and \ref{e2.2}. For a small $\eps\in\bigl(0,c/(4\Lip(h))\bigr)$ consider
the function $g:=f+\eps h$. If $x\in E$, \ref{e2.1}
shows that for some $y\in \R$,
${D^+}g(x;y)-{D_+}g(x;y)
\ge (c-2\eps\Lip(h))\|y\|>0$. If $x\in D\setminus E$,
$g$ is the sum of the function $f$ that is differentiable at $x$
and of the function $\eps h$ that is non-differentiable at $x$;
hence it is non-differentiable at $x$. Consequently, the Lipschitz function $g$ has no
point of differentiability at $D$, contradicting that $D$ is a
universal differentiability set.
\end{proof}

\begin{rmk}
The reason for considering a uniform non-differentiability condition such as~\ref{e2.1} was explained in the text before the Example. Notice that, if~\ref{e2.1} were replaced just by non-differentiability of $f$ at every point of~$E$, the statement of the Example would be false: we would use Theorem~\ref{9} to find a function $g$ that is non-differentiable at every point of $A$ and define $E$ as the non-differentiability set of $g$.
On the other hand, it is easy to find uniformly purely unrectifiable sets $E\supset A$ for which there is no Lipschitz function non-differentiable exactly at points of $E$, as such $E$ need not be~$G_{\delta\sigma}$. For the set $A$ from \cite{CPT} which was used in the proof of the Example~\ref{e2} we can take $E=A$ as it is not difficult to see that $A$
is not $G_{\delta\sigma}$, although it is $F_{\sigma\delta}$ since $A$
is the intersection of $D$ with the set of points of differentiability
of~$h$ and $D$ used in \cite{CPT} is $G_\delta$.
It may be of interest to notice that
the fact that $A$ is not a non-differentiability
set of any Lipschitz function $f$ may be seen
directly from the properties of $A$, $D$
and $h$: for any such $f$ the Lipschitz function $f+h$ would be non-differentiable at any $x\in D\setminus A$ as $f$ is differentiable and $h$ is not differentiable at such $x$; and $f+h$ would be non-differentiable at any $x\in A$ as $f$ is not differentiable and $h$ is differentiable at such $x$. As in the proof of the Example~\ref{e2}, this a contradiction as $D$ is a universal differentiability set.
\end{rmk}

Our final example is related to the already pointed out fact that
$E$ is uniformly purely unrectifiable if and only if
there is $0<\eta<1$ such that $w_{e,\eta}(E)=0$ for every
unit vector~$e$. When considering general non-differentiability sets,
a natural
analogy of this statement would say that for any set $E\subset\R^n$
satisfying $w_{e,\eta}(E)=0$ for some unit vector $e$ and some
$0<\eta<1$ there is
a real-valued Lipschitz
function $f$ on $\R^n$ that is non-differentiable
at any point of $E$. We show here that this is false; recall however
that \cite{ACP} shows (directly, not using \cite{CJ})
that for any such set $E$ there is
an $\R^n$-valued Lipschitz
function $f$ on $\R^n$ that is non-differentiable
at any point of $E$.

\begin{exm}\label{e4}
For every $\eta\in(0,1)$ and a unit vector $e\in\R^2$ there is
a universal differentiability set $E\subset\R^2$ such that
$w_{e,\eta}(E)=0$.
\end{exm}

\begin{proof}
Let $L_j$ be an enumeration of all rational lines in $\R^n$,
$J$ the set
of those indexes $j$ for which the direction $u$
of $L_j$ satisfies $|\spr{u}{e}|<\tfrac12\eta$ and
$\eps_{i,j}>0$ such that $\sum_{i,j} \eps_{i,j}<\infty$. It
is easy to see that $E:= \bigcap_j\bigcup_{j\in J}\{x:\dist(x,L_j)<\eps_{i,j}\}$
satisfies $w_{e,\eta}(E)=0$.
The fact that $E$
is a universal differentiability set has been often mentioned, but
does not seem to be documented in the literature.
We therefore explain the argument.

Recall from
\cite{MDo}, \cite{MD} or \cite{p} that, given any Lipschitz $g:\R^n\to\R$,
a procedure leading to a point of differentiability of $g$ may
be described as follows.
One starts with an arbitrary $\delta_0>0$ and $(x_0,e_0)$ from the set
$D$ of pairs $(x,u)$ where $x\in\R^n$, $u$ is a unit vector, and there is
$j=j(x,u)$ such that $x\in L_j$ and $u$ is the direction of~$L$. Recursively,
when $(x_k,e_k)$ has been defined, one
first chooses an arbitrarily small $\delta_{k+1}>0$
and then $(x_{k+1},e_{k+1})\in D$ satisfying
rather delicate conditions about which we need to know only
that $x_{k+1}\in B(x_k,\delta_{k+1})$,
$D g(x_{k+1},e_{k+1})\ge D g(x_{k},e_{k})$
and that they imply that
the sequence $x_k$ converges to a point of differentiability of~$g$.

Returning to our set $E$, given any Lipschitz $f:\R^n\to\R$,
choose $(x_0,e_0)\in D$ so that $|\spr{e_0}{e}|<\tfrac14\eta$ and
let $g(x):=f(x)+c\spr{x}{e_0}$ with $c>64\Lip(f)/\eta^2$; the choice of such large $c$ guarantees that
$D g(x,u) \ge D g(x_0;e_0)$ implies $0\le1-\spr{u}{e_0}\le\frac1c(D f(x;u)-D f(x_0;e_0))\le2\Lip(f)/c\le\frac1{32}\eta^2$, so that $\|u-e_0\|\le\frac14\eta$, hence $|\spr{u}{e}|\le\|u-e_0\|+|\spr{e_0}{e}|<\tfrac12\eta$.
This will imply that in the recursive construction
$j_k:=j(x_k,e_k)\in J$, and so we can choose
$\delta_{k+1}$ such that
$\Bc(x_k,\delta_{k+1})\subset B(L_{j_k},\eps_{k,j_k})\cap B(x_k,\delta_k)$.
Hence the limit of the $x_k$, which is
a differentiability point of $g$ and so of $f$, belongs to $E$.
\end{proof}

\end{document}